\documentclass{amsart}
\usepackage{amssymb}
\usepackage{multicol}

\allowdisplaybreaks

\newtheorem{theorem}{Theorem}[section]
\newtheorem{proposition}[theorem]{Proposition}
\newtheorem{conjecture}[theorem]{Conjecture}
\newtheorem{corollary}[theorem]{Corollary}
\newtheorem{example}[theorem]{Example}
\newtheorem{lemma}[theorem]{Lemma}
\numberwithin{equation}{section}

\begin{document}

\title[On Automorphisms of quantum Schubert cells]{On Automorphisms of quantum
Schubert cells}

\author{Garrett Johnson}

\address{Department of Mathematics and Physics \\North Carolina Central
University\\Durham, NC 27707\\USA}

\email{gjohns62@nccu.edu}

\author{Hayk Melikyan}

\address{Department of Mathematics and Physics \\North Carolina Central
University\\Durham, NC 27707\\USA}

\email{gmelikian@nccu.edu}



\thanks{The authors were supported by NSF grant DMS-1900823.}

\begin{abstract}

    Automorphisms of the quantum Schubert cell algebras ${\mathcal U}_q^\pm[w]$
    of De Concini, Kac, Procesi \cite{DKP} and Lusztig \cite{L} and their
    restrictions to some key invariant subalgebras are studied. We develop some
    general rigidity results and apply them to completely determine the
    automorphism group in several cases.

    We focus primarily on those cases when the underlying Lie algebra
    $\mathfrak{g}$ is finite dimensional and simple with rank $r > 1$, and $w$
    is a parabolic element of the Weyl group, say $w = w_o^Jw_o$, for some
    nonempty subset $J$ of simple roots. Here, ${\mathcal U}_q^\pm[w]$
    is a deformation of the universal enveloping algebra of the
    nilradical of a parabolic subalgebra of $\mathfrak{g}$.  In this setting we
    conjecture that, with the exception of two specific low rank cases, the
    automorphism group of ${\mathcal U}_q^{\pm}[w]$ is the semidirect product
    of an algebraic torus of rank $r$ with the group of Dynkin diagram
    symmetries that preserve $J$. This conjecture is a more general form of the
    Launois-Lenagan \cite{LL} and Andruskiewitsch-Dumas \cite{AD} conjectures
    regarding the automorphism groups of the algebras of quantum matrices and
    the algebras ${\mathcal U}_q^+(\mathfrak{g})$, respectively.  We completely
    determine the automorphism group in several instances, including all cases when
    $\mathfrak{g}$ is of type $F_4$ or $G_2$, as well as those cases when the
    quantum Schubert cell algebras are the algebras of quantum symmetric
    matrices.

\end{abstract}

\maketitle

\section{Introduction and summary of the results}

Quantum Schubert cell algebras ${\mathcal U}_q^\pm[w]$ were introduced by De
Concini, Kac, Procesi \cite{DKP} and Lusztig \cite{L}. They are a family of
subalgebras of the Drinfeld-Jimbo quantized enveloping algebra ${\mathcal
U}_q(\mathfrak{g})$ indexed by the elements $w$ of the Weyl group, and have
appeared in several contexts, including ring theory \cite{MC, Y}, crystal basis
theory \cite{Kashiwara, Lusztig}, and cluster algebras \cite{GLS, Goodearl
Yakimov}. Several important cases arise when the Weyl group element $w$ is a
parabolic element, say $w_o^Jw_o$, for some nonempty subset $J$ of simple
roots. Here, the corresponding algebra ${\mathcal U}_q^\pm[w]$ can be viewed as
a deformation of the universal enveloping algebra of the nilradical
$\mathfrak{n}_J$ of a parabolic subalgebra of $\mathfrak{g}$.  In such a
setting, we denote the quantum Schubert cell algebra by ${\mathcal
U}_q(\mathfrak{n}_J)$ and we refer to it as a \textit{quantized nilradical} for
short.

Throughout, the underlying base field for all algebras will be denoted by
$\mathbb{K}$. We do not need to assume that $\mathbb{K}$ is algebraically
closed or that it is of characteristic zero. The role of the Lie algebra
$\mathfrak{g}$ in defining the $\mathbb{K}$-algebras ${\mathcal U}_q^\pm[w]$,
${\mathcal U}_q \left( \mathfrak{n}_J \right)$, and ${\mathcal U}_q\left(
\mathfrak{g} \right)$ can be viewed as purely symbolic. We will denote the
multiplicative group of nonunits by $\mathbb{K}^\times$.

We will turn our attention towards studying the automorphisms of these
algebras. We assume that the deformation parameter $q \in \mathbb{K}^\times$ is
not a root of unity. Basically, there is a dichotomy in the structure of
quantum algebras depending on whether or not $q$ is a root of unity. When $q$
is a root of unity, these algebras more closely resemble deformations of
modular Lie algebras (Lie algebras over fields of positive characteristic).
These algebras have large centers, and are therefore closer to being
commutative. This, in effect, gives less control over automorphisms.  In such
situations, various types of noncommutative discriminants \cite{CPWZ1, CPWZ2,
CGWZ} have been developed as tools to study automorphisms (see e.g. \cite{GKM,
GL}).

Automorphism groups of quantized nilradicals (when $q$ is not a root of unity)
have already been studied in several cases. For instance, when the underlying
Lie algebra is $\mathfrak{sl}(n)$ and $J$ is a singleton, say $J =
\left\{\alpha_k\right\}$, the quantized nilradical ${\mathcal
U}_q(\mathfrak{n}_J)$ is isomorphic to the algebra of quantum $k \times (n -
k)$ matrices.  Launois and Lenagan prove in \cite{LL} that the automorphism
group is $\left(\mathbb{K}^\times \right)^{n-1}$ whenever $k \neq n - k$ (and
when $(n,k) \not\in \left\{(4,1), (4,3)\right\}$) by using certain properties
of height one prime ideals.  These techniques do not apply when $k = n - k$,
yet they conjecture that the automorphism group in this remaining case is
$\left(\mathbb{K}^\times \right)^{n-1} \rtimes \mathbb{Z}_2$.  Their conjecture
was already known to be true in the $2 \times 2$ case by the work of Alev and
Chamarie \cite{AC}.  Launois and Lenagan later proved their conjecture for the
$3\times 3$ case in \cite{LL2}. Finally, the Launois-Lenagan conjecture was
proved in the remaining cases by Yakimov in \cite{Y4}.

An interesting phenomenon regarding automorphisms arises when $k = 1$ or $n - k
= 1$. In this setting ${\mathcal U}_q\left(\mathfrak{n}_J\right)$ is isomorphic
to $(n-1)$-dimensional quantum affine space ${\mathcal A}_q \left(
\mathbb{K}^{n-1} \right)$. As an algebra, ${\mathcal A}_q\left( \mathbb{K}^{n -
1} \right)$ is generated by elements $x_1,\dots,x_{n-1}$ and has defining
relations $x_ix_j = qx_jx_i$ whenever $i < j$.  In \cite{AC}, Alev and Chamarie
studied automorphisms of several types of noncommutative algebras, including
multiparameter and uniparameter quantum affine space. Their work predates the
Launois-Lenagan conjecture.  Interestingly, automorphisms of ${\mathcal
A}_q\left( \mathbb{K}^{n-1} \right)$ send each generator $x_i$ to a scalar
multiple of itself whenever $n \neq 4$.  Alev and Chamarie proved in
\cite[Theorem 1.4.6]{AC} that every automorphism $\phi$ of ${\mathcal
A}_q\left(\mathbb{K}^3\right)$ has the form \[ \phi(x_1) = a_1x_1,
\hspace{10mm} \phi(x_2) = a_2x_2 + bx_1x_3, \hspace{10mm} \phi(x_3) = a_3x_3,
\] where $a_1, a_2, a_3 \in \mathbb{K}^\times$ and $b\in \mathbb{K}$.  Hence,
the automorphism group of ${\mathcal A}_q(\mathbb{K}^3)$ is isomorphic to the
semidirect product $\left( \mathbb{K}^\times \right)^3 \rtimes \mathbb{K}$.  On
the other hand, if $n \neq 4$, they proved that the automorphism group of
${\mathcal A}_q(\mathbb{K}^{n-1})$ is isomorphic to $\left( \mathbb{K}^\times
\right)^{n-1}$. Here, every automorphism sends $x_i$ to a nonzero multiple of
itself.

Automorphism groups of quantized nilradicals have also been determined in all
cases when $J$ is chosen to be the full set of simple roots. In particular, we
assume now that $\mathfrak{g}$ is an arbitrarily chosen finite dimensional
complex simple Lie algebra with $\operatorname{rank}(\mathfrak{g}) = r > 1$,
and $J$ is the full set of simple roots. Here, ${\mathcal U}_q
\left(\mathfrak{n}_J \right)$ is the entire positive part of ${\mathcal U}_q
\left(\mathfrak{g} \right)$. The Chevalley generators $E_1,\dots, E_r$ generate
${\mathcal U}_q^+ \left( \mathfrak{g} \right)$ as an algebra and satisfy the
$q$-Serre relations. With this, it is not too difficult to observe that for
every $r$-tuple $(a_1,\dots,a_r) \in \left( \mathbb{K}^\times \right)^r$, there
is an algebra automorphism $\phi$ of ${\mathcal U}_q^+ \left( \mathfrak{g}
\right)$ such that $\phi(E_i) = a_iE_i$ ($i=1,\dots,r$). Furthermore, for every
symmetry $\psi$ of the underlying Dynkin diagram, there is an algebra
automorphism of ${\mathcal U}_q^+ \left(\mathfrak{g} \right)$ given by the rule
$E_i \mapsto E_{\psi(i)}$.  Andruskiewitsch and Dumas \cite{AD} conjectured
that the automorphism group of ${\mathcal U}_q^+ \left(\mathfrak{g} \right)$ is
generated by only these types of automorphisms. That is to say, they
conjectured that
\[
    \operatorname{Aut} \left( {\mathcal U}_q^+ \left( \mathfrak{g} \right)
    \right) \cong \left( \mathbb{K}^\times \right)^{\operatorname{rank} \left(
    \mathfrak{g} \right)} \rtimes \operatorname{Dynkin-Aut} \left( \mathfrak{g}
    \right),
\]
where $\operatorname{Dynkin-Aut} \left( \mathfrak{g} \right)$ is the
automorphism group of the Dynkin diagram of $\mathfrak{g}$. Yakimov proved this
conjecture in \cite{Y3} using a rigidity result involving quantum tori.

In describing the automorphism group ${\mathcal U}_q \left( \mathfrak{n}_J
\right)$ (for arbitrarily chosen $\mathfrak{g}$ and $J$), we need to introduce
the subgroup of Dynkin diagram symmetries that fixes $J$,
\[
    \operatorname{Dynkin-Aut}_J(\mathfrak{g}) := \left\{ \psi \in
    \operatorname{Dynkin-Aut}(\mathfrak{g}) \mid \psi (J) = J \right\}.
\]
We conjecture the following result regarding the automorphism groups of
quantized nilradicals.

\begin{conjecture}
    \label{conj}

    Let $\mathfrak{g}$ be a finite dimensional complex simple Lie algebra
    with $\operatorname{rank}(\mathfrak{g}) > 1$. Suppose $J$ is a nonempty
    subset of simple roots, and let ${\mathcal U}_q(\mathfrak{n}_J)$ be the
    corresponding quantized nilradical. Then
    \[
        \operatorname{Aut}\left({\mathcal U}_q \left( \mathfrak{n}_J \right)
        \right) \cong \left( \mathbb{K}^\times
        \right)^{\operatorname{rank}\left( \mathfrak{g} \right)} \rtimes
        \operatorname{Dynkin-Aut}_J(\mathfrak{g}).
    \]
    provided ${\mathcal U}_q \left( \mathfrak{n}_J \right) \not\cong
    \mathcal{A}_q(\mathbb{K}^3)$.

\end{conjecture}

\noindent We remark that ${\mathcal U}_q \left( \mathfrak{n}_J \right)$ is
isomorphic to ${\mathcal A}_q(\mathbb{K}^3)$ in only two situations: (1)
$\mathfrak{g} = \mathfrak{sl}(4)$ and $J = \left\{ \alpha_1 \right\}$, or (2)
$\mathfrak{g} = \mathfrak{sl}(4)$ and $J = \left\{ \alpha_3 \right\}$.

Conjecture \ref{conj} above has been resolved in several cases. As mentioned
above, the proof of the Launois-Lenagan conjecture covers the situation when
$\mathfrak{g}$ is of type $A_n$ and $J$ is a singleton, whereas the
Andruskiewitsch-Dumas conjecture handles the case when $\mathfrak{g}$ is
arbitrary and $J$ is the full set of simple roots.  We prove Conjecture
\ref{conj} in some other situations, including when the underlying Lie algebra
$\mathfrak{g}$ is of type $F_4$ or $G_2$.

\begin{theorem}

    If $\mathfrak{g}$ is the Lie algebra of type $F_4$ and $J$ is a nonempty
    subset of simple roots of $\mathfrak{g}$, then
    $\operatorname{Aut}({\mathcal U}_q(\mathfrak{n}_J)) \cong
    \left(\mathbb{K}^\times\right)^4$.

\end{theorem}

\begin{theorem}

    If $\mathfrak{g}$ is the Lie algebra of type $G_2$ and $J$ is a nonempty
    subset of simple roots of $\mathfrak{g}$, then
    $\operatorname{Aut}({\mathcal U}_q(\mathfrak{n}_J)) \cong
    \left(\mathbb{K}^\times\right)^2$.

\end{theorem}

We also develop some general theorems (Theorems \ref{when Aut is Gr-Aut},
\ref{C gamma, fixed}, \ref{normal gen, fixed}, \ref{root vector, fixed},
\ref{root vector, fixed, 2}) regarding automorphisms of quantum Schubert cell
algebras that can be applied to help determine the automorphism groups of
several other quantized nilradicals. More generally, quantum Schubert cell
algebras belong to a larger family of algebras called Cauchon-Goodearl-Letzter
(CGL) extensions, which originated in the works \cite{C, Goodearl Letzter}.
Some general techniques have been developed in \cite{GY1} to study
automorphisms of CGL extensions. These techniques utilize properties of some
key subalgebras of a CGL extension $R$, namely the normal subalgebra ${\mathcal
N}(R)$ (the subalgebra generated by the normal elements), and the core
${\mathcal C}(R)$. Basically, the larger the core ${\mathcal C}(R)$, the more
control one has over automorphisms \cite[Theorem 4.2]{GY2}. One has the most
control over automorphisms when the core coincides with the entire algebra.
Most quantized nilradicals appear to have this property.

Several other instances of quantized nilradicals appear in the literature,
particularly when the nilradical $\mathfrak{n}_J$ is abelian. For example, when
the underlying Lie algebra $\mathfrak{g}$ is of type $C_n$ and $J = \left\{
\alpha_n \right\}$, the corresponding quantized nilradical ${\mathcal
U}_q(\mathfrak{n}_J)$ is the algebra of quantum $n \times n$ symmetric matrices
\cite{Kamita, Noumi}. When $\mathfrak{g}$ is the Lie algebra of type $D_n$ and
$J = \left\{\alpha_{n - 1}\right\}$ or $J = \left\{\alpha_n\right\}$,
${\mathcal U}_q(\mathfrak{n}_J)$ is the algebra of quantum antisymmetric
matrices \cite{Strickland}.  If $\mathfrak{g}$ is of type $B_n$ and $J =
\left\{ \alpha_1 \right\}$, ${\mathcal U}_q(\mathfrak{n}_J)$ is the
odd-dimensional quantum Euclidean space, which was introduced by Faddeev,
Reshetikhin, and Takhtadzhyan \cite[Definition 12]{FRT}. Simplified relations
for this algebra appear in \cite[Sections 2.1-2.2]{Musson}.  If $\mathfrak{g}$
is of type $D_n$ and $J = \left\{ \alpha_1 \right\}$, ${\mathcal
U}_q(\mathfrak{n}_J)$ is the even-dimensional quantum Euclidean space
\cite{FRT, Musson}. The automorphism groups of even and odd-dimensional quantum
Euclidean space are already known to satisfy Conjecture \ref{conj}
\cite[Example 4. 10]{GY1}.  We prove Conjecture \ref{conj} holds when
${\mathcal U}_q(\mathfrak{n}_J)$ is the algebra of quantum symmetric matrices.

\begin{theorem}

    If $\mathfrak{g}$ is the Lie algebra of type $C_n$ with $n > 1$ and $J =
    \left\{\alpha_n\right\}$ (i.e. ${\mathcal U}_q(\mathfrak{n}_J)$ is the
    algebra of $n\times n$ quantum symmetric matrices), then
    $\operatorname{Aut}({\mathcal U}_q(\mathfrak{n}_J)) \cong
    \left(\mathbb{K}^\times\right)^n$.

\end{theorem}

Other examples of quantized nilradicals ${\mathcal U}_q(\mathfrak{n}_J)$ have
been studied for cases when $\mathfrak{n}_J$ is non-abelian.  For instance, the
quantized nilradicals when $\mathfrak{g}$ is of type $A_n$ and $J$ is an
arbitrary set of simple roots were studied in \cite{JJ}, where it was shown
that ${\mathcal U}_q(\mathfrak{n}_J)$ is isomorphic to an algebra of
coinvariants. With this, ${\mathcal U}_q(\mathfrak{n}_J)$ can be viewed as a
deformation of the coordinate ring of a unipotent subgroup of a parabolic
subgroup of $SL(n + 1)$.

Each algebra ${\mathcal U}_q(\mathfrak{n}_J)$ can be equipped with a
$\mathbb{N}$-grading such that, with respect to this grading, ${\mathcal
U}_q(\mathfrak{n}_J)$ is connected and locally finite.  We apply the results
developed in Theorems \ref{when Aut is Gr-Aut}, \ref{C gamma, fixed},
\ref{normal gen, fixed}, \ref{root vector, fixed}, \ref{root vector, fixed, 2}
to illustrate that, for certain cases of $\mathfrak{g}$ and $J$, every
automorphism of ${\mathcal U}_q(\mathfrak{n}_J)$ that preserves the
$\mathbb{N}$-grading acts diagonally on the graded component of degree one (see
Proposition \ref{prop}).

We choose the first case listed in Proposition \ref{prop}, namely when
$\mathfrak{g}$ is of type $B_6$ and $J = \left\{ \alpha_2, \alpha_5 \right\}$,
and completely determine the automorphism group of the corresponding quantized
nilradical. We show here that every automorphism preserves the
$\mathbb{N}$-grading by applying the results of \cite[Theorem 4.2]{GY2}
involving the core of ${\mathcal U}_q(\mathfrak{n}_J)$. The same steps can be
applied to other cases listed in Proposition \ref{prop}.

\begin{theorem}

    If $\mathfrak{g}$ is the Lie algebra of type $B_6$ and $J =
    \left\{\alpha_2, \alpha_5\right\}$, then $\operatorname{Aut}({\mathcal
    U}_q(\mathfrak{n}_J)) \cong \left(\mathbb{K}^\times\right)^6$.

\end{theorem}

Ideally, we would like to eventually develop a theory sufficient to completely
determine the automorphism group of ${\mathcal U}_q \left( \mathfrak{n}_J
\right)$ in all cases. A more general endeavor is to develop a theory
sufficient to describe the automorphism groups of quantum Schubert cell
algebras ${\mathcal U}_q^{\pm}[w]$. Interestingly, Ceken, Palmieri, Wang, and
Zhang \cite{CPWZ1} describe a family of algebras such that the automorphism
group of each algebra in this family is isomorphic to the semidirect product of
an algebraic torus and a finite group.  While quantum Schubert cells don't
belong to this family of algebras, in many instances their automorphism groups
seem to have this form.  In a related work, one could attempt to find necessary
and sufficient conditions on $w$ so that $\operatorname{Aut} \left( {\mathcal
U}_q^{\pm}[w] \right)$ is isomorphic to $\left( \mathbb{K}^\times \right)^n
\rtimes G$ for some natural number $n\in \mathbb{N}$ and finite group $G$.

\section{The algebra $\mathbf{{\mathcal U}_q(\mathfrak{g})}$}

Let $\mathfrak{g}$ be a finite dimensional complex simple Lie algebra of rank
$r$. Define the index set $\mathbf{I} := \left\{1,2,\dots r\right\}$, and let
$\Pi = \{\alpha_i\}_{i \in \mathbf{I}}$ be a set of simple roots of
$\mathfrak{g}$ with respect to a fixed Cartan subalgebra $\mathfrak{h} \subset
\mathfrak{g}$ such that the labelling of the simple roots agrees with the
labelling in \cite[Section 12.1]{Humphreys}.

The root system of $\mathfrak{g}$ will be denoted by $\Delta$, and the sets of
positive and negative roots will be denoted by $\Delta_+$ and $\Delta_-$,
respectively.  The corresponding triangular decomposition of $\mathfrak{g}$
will be denoted by
\[
    \mathfrak{g} = \mathfrak{n}^- \oplus \mathfrak{h} \oplus \mathfrak{n}^+,
\]
where
\[
    \mathfrak{n}^{\pm} := \bigoplus_{\alpha \in \Delta_\pm}
    \mathfrak{g}_\alpha, \hspace{10mm} \mathfrak{g}_\alpha := \left\{ x \in
    \mathfrak{g} \mid [h,x] = \alpha(h)x \text{ for all }h \in \mathfrak{h}
    \right\}.
\]
As usual, let $Q = \oplus_{i \in \mathbf{I}} \mathbb{Z}\alpha_i$ be the root
lattice of $\mathfrak{g}$, and let $\langle \,\, ,\, \rangle : Q \times Q \to
\mathbb{Z}$ be a symmetric nondegenerate ad-invariant $\mathbb{Z}$-bilinear
form, normalized so that $\langle \alpha, \alpha \rangle = 2$ for short roots
$\alpha$. Define
\[
    c_{ij} := \frac{2 \langle \alpha_i , \alpha_j \rangle}{\langle \alpha_i,
    \alpha_i \rangle}, \hspace{5mm} (i,j \in \mathbf{I}),
\]
and let $(c_{ij})_{i,j \in \mathbf{I}}$ be the associated Cartan matrix of
$\mathfrak{g}$. We will denote the simple reflections in the Weyl group $W$ of
$\mathfrak{g}$ by
\[
    s_i, \hspace{5mm} (i \in \mathbf{I}).
\]
The corresponding generators of the braid group ${\mathcal B}_{\mathfrak{g}}$
of $\mathfrak{g}$ will be denoted by
\[
    T_i, \hspace{5mm} (i \in \mathbf{I}).
\]

Let $q \in \mathbb{K}$ be nonzero and not a root of unity, and let $q_i =
q^{\langle \alpha_i, \alpha_i \rangle/ 2}$ for $i \in \mathbf{I}$. As usual,
define
\[
        \widehat{q} := q - q^{-1}.
\]
For a natural number $n\in \mathbb{N}$, define
\[
    [n]_{q_i} := \frac{q_i^n - q_i^{-n}}{q_i - q_i^{-1}}, \hspace{10mm}
    [n]_{q_i}! := [n]_{q_i} [n - 1]_{q_i} \cdots [1]_{q_i}.
\]
The quantized universal enveloping algebra ${\mathcal U}_q(\mathfrak{g})$ is an
associative $\mathbb{K}$-algebra with standard Chevalley generators
\[
    K_\mu, E_i, F_i, \hspace{5mm} (\mu \in Q, i \in \mathbf{I}).
\]
There is a standard $Q$-gradation on the algebra ${\mathcal
U}_q(\mathfrak{g})$,
\[
    {\mathcal U}_q(\mathfrak{g}) = \bigoplus_{\lambda \in Q} {\mathcal
    U}_q(\mathfrak{g})_\lambda, \hspace{5mm} {\mathcal U}_q(\mathfrak{g})_\mu
    := \left\{ u\in {\mathcal U}_q(\mathfrak{g}) : K_\lambda u = q^{\langle
    \lambda, \mu \rangle}uK_\lambda \text{ for all } \lambda \in Q \right\}.
\]
With respect to this grading, the Chevalley generators are homogeneous
elements. In particular,
\[
    E_i \in {\mathcal U}_q(\mathfrak{g})_{\alpha_i}, \hspace{5mm} F_i \in
    {\mathcal U}_q(\mathfrak{g})_{-\alpha_i}, \hspace{5mm} K_\mu \in {\mathcal
    U}_q(\mathfrak{g})_0, \hspace{10mm} (i\in \mathbf{I}, \mu \in Q).
\]
We will state the defining relations of ${\mathcal U}_q(\mathfrak{g})$, but
first we find it convenient to introduce the abbreviation
\begin{equation}
    \label{def, q-commutators}
    [x, y] := x y - q^{\langle \mu, \eta \rangle} y x, \hspace{10mm} x\in
    {\mathcal U}_q(\mathfrak{g})_{\mu}, \hspace{3mm} y \in {\mathcal
    U}_q(\mathfrak{g})_{\eta},
\end{equation}
for \textit{$q$-commutators}. We will adopt this notation throughout.  Next,
for every homogeneous $x \in {\mathcal U}_q(\mathfrak{g})_\mu$, we define the
linear operator $\operatorname{ad}_qx:{\mathcal U}_q(\mathfrak{g}) \to
{\mathcal U}_q(\mathfrak{g})$ by the condition that
\[
    \left(\operatorname{ad}_q x\right)(y) = [x, y],
\]
for every homogeneous element $y \in {\mathcal U}_q(\mathfrak{g})$.  With this,
the defining relations of ${\mathcal U}_q(\mathfrak{g})$ are
\[
    \begin{split}
        &K_0 = 1, \hspace{10mm} K_\mu K_\lambda = K_{\lambda + \mu},
        \\
        &K_\mu E_i = q^{\langle \mu, \alpha_i \rangle} E_i K_\mu, \hspace{10mm}
        K_\mu F_i = q^{-\langle \mu, \alpha_i \rangle} F_i K_\mu,
        \\
        &E_iF_j - F_jE_i = \delta_{ij} \frac{K_{\alpha_i} - K_{-\alpha_i}}{q_i
            - q_i^{-1}}
    \end{split}
\]
\noindent together with the $q$-Serre relations
\[
    \begin{split}
        &(\operatorname{ad}_qE_i)^{1 - c_{ij}} (E_j) = 0, \hspace{10mm}
        (\text{for all $i \neq j$}),
        \\
        &(\operatorname{ad}_qF_i)^{1 - c_{ij}} (F_j) = 0, \hspace{10mm}
        (\text{for all $i \neq j$}).
    \end{split}
\]
The algebra ${\mathcal U}_q(\mathfrak{g})$ has a triangular decomposition,
\[
    {\mathcal U}_q(\mathfrak{g}) \cong {\mathcal U}_q(\mathfrak{n}^-) \otimes
    {\mathcal U}_q(\mathfrak{h}) \otimes {\mathcal U}_q(\mathfrak{n}^+),
\]
where ${\mathcal U}_q(\mathfrak{n}^-)$, ${\mathcal U}_q(\mathfrak{h})$, and
${\mathcal U}_q(\mathfrak{n}^+)$ are the subalgebras of ${\mathcal
U}_q(\mathfrak{g})$ generated by the $F$'s, $K$'s, and $E$'s respectively.

\subsection{A $\mathbb{Z}$-grading on ${\mathcal U}_q(\mathfrak{g})$}\label{Z
grading section} The fundamental coweights of $\mathfrak{g}$ will be denoted by
$\varpi_i^\vee$ ($i\in \mathbf{I}$). They are determined by the conditions
$\langle \varpi_i^\vee, \alpha_j\rangle = \delta_{ij}$ for all $i,j\in
\mathbf{I}$. We will let $P^\vee = \oplus_{i\in
\mathbf{I}}\mathbb{Z}\varpi_i^\vee$ be the coweight lattice of $\mathfrak{g}$.
For every integral coweight $\lambda\in P^\vee$ (i.e.  $\langle \lambda,
\alpha_i \rangle \in \mathbb{Z}$ for every $i \in \mathbf{I}$), there is an
associated $\mathbb{Z}$-grading on ${\mathcal U}_q(\mathfrak{g})$ given by
assigning the degree $\langle \lambda, \mu \rangle$  to every $Q$-homogeneous
element of degree $\mu$.  To make the distinction between the $Q$-grading and
the $\mathbb{Z}$-grading, we will write
\[
    \operatorname{deg}_Q(u) = \mu, \hspace{5mm}
    \operatorname{deg}_\mathbb{Z}(u) = n,
\]
respectively, to mean that $u \in {\mathcal U}_q(\mathfrak{g})$ is a
$Q$-homogeneous element of degree $\mu \in Q$ and a $\mathbb{Z}$-homogeneous
element of degree $n$.  Technically, $\operatorname{deg}_\mathbb{Z}$ depends on
the coweight $\lambda$. However, we adopt the notation
$\operatorname{deg}_\mathbb{Z}$ rather than $\operatorname{deg}_\lambda$
whenever the choice of coweight $\lambda$ is clear from the context.

\subsection{Lusztig symmetries of $\mathbf{{\mathcal U}_q(\mathfrak{g})}$} In
\cite[Section 37.1.3]{L}, Lusztig defines an action of the braid group
${\mathcal B}_{\mathfrak{g}}$ via algebra automorphisms on ${\mathcal
U}_q(\mathfrak{g})$.  In fact, Lusztig defines the symmetries $T_{i,1}^\prime$,
$T_{i,-1}^{\prime}$, $T_{i,1}^{\prime\prime}$, and $T_{i,-1}^{\prime\prime}$.
By \cite[Proposition 37.1.2]{L}, these are automorphisms of ${\mathcal
U}_q(\mathfrak{g})$, while by \cite[Theorem 39.4.3]{L} they satisfy the braid
relations.  For short, we will adopt the abbreviation $T_i :=
T^{\prime\prime}_{i,1}$. With this convention, Lusztig's symmetries are given
by the formulas
\[
    \begin{split}
        &T_i(K_\mu) = K_{s_i(\mu)},
        \\
        &T_i(E_j) =
        \begin{cases}
            - F_i K_{\alpha_i}, &
            (i = j),
            \\
            \left(\operatorname{ad}_q E_i \right)^{(-c_{ij})} (E_j), &
            (i \neq j),
            \\
        \end{cases}
        \\
        &T_i(F_j) =
        \begin{cases}
            - K_{-\alpha_i} E_i, &
            (i = j),
            \\
            \left(-q_i\right)^{-c_{ij}} \left(\operatorname{ad}_q
            F_i\right)^{(-c_{ij})}(F_j), &
            (i \neq j),
        \end{cases}
    \end{split}
\]
\noindent where, for a nonnegative integer $n$,
\[
    \left(\operatorname{ad}_q E_i \right)^{(n)} := \frac{ \left(
            \operatorname{ad}_q E_i\right)^n}{[n]_{q_i}!},
            \hspace{10mm}
    \left(\operatorname{ad}_q F_i \right)^{(n)} := \frac{ \left(
            \operatorname{ad}_q F_i\right)^n}{[n]_{q_i}!}.
\]
If $w \in W$ has a reduced expression $w = s_{i_1} \cdots s_{i_N} \in W$, we
write
\[
    T_w = T_{i_1} T_{i_2} \cdots T_{i_N}.
\]
A key property of the braid symmetries is given in the following proposition
(see e.g. \cite[Proposition 8.20]{Jantzen}).

\begin{proposition}

    \label{proposition 1}

    If $w\in W$ such that $w(\alpha_i) = \alpha_j$, then $T_w(E_i) = E_j$.

\end{proposition}

\section{The quantized nilradical ${\mathcal U}_q(\mathfrak{n}_J)$}

For each nonempty set $J$ of simple roots, let $\mathfrak{p}_J$ be the
parabolic subalgebra of $\mathfrak{g}$ obtained by deleting the
roots in $J$. The Levi decomposition of $\mathfrak{p}_J$ will be denoted by
\[
    \mathfrak{p}_J = \mathfrak{l}_J \ltimes \mathfrak{n}_J,
\]
where $\mathfrak{l}_J$ is the Levi subalgebra and $\mathfrak{n}_J$ is the
nilradical.

Let $w_o$ denote the longest element of the Weyl group, and
let $w_o^J\in W$ be the longest element in the subgroup $\langle s_i \mid i
\not\in J \rangle \subseteq W$. Define
\[
    w_J := w_o^Jw_o \in W.
\]
For a reduced expression,
\[
    w_J = s_{i_1} \cdot s_{i_2} \cdots s_{i_N} \in W,
\]
where $N$ is the length of $w_J$, define the roots
\[
    \beta_1 = \alpha_{i_1}, \beta_2=s_{i_1}\alpha_{i_2},...,\beta_N =
    s_{i_1}\cdots s_{i_{N-1}}\alpha_{i_N},
\]
and root vectors
\begin{equation}
    \label{root vectors}
    X_{\beta_1} = E_{i_1}, X_{\beta_2} = T_{s_{i_1}} E_{i_2},..., X_{\beta_N} =
    T_{s_{i_1}} \cdots T_{s_{i_{N-1}}} E_{i_N}.
\end{equation}

We will denote the set of radical roots by
\[
    \Delta_w := \left\{ \beta_1, \dots, \beta_N\right\}.
\]
These roots are precisely the positive roots that get sent to negative roots by
the action of $w_J^{-1}$. An analogous construction can be applied to obtain a
list of negative roots by replacing the $E$'s in (\ref{root vectors}) above
with $F$'s. The subalgebra of ${\mathcal U}_q(\mathfrak{g})$ generated by the
root vectors $X_{\beta_1},\dots, X_{\beta_N}$ is contained in the positive part
${\mathcal U}_q(\mathfrak{n}^+)$ (see e.g. \cite[Proposition 8.20]{Jantzen}).
This subalgebra will be denoted by ${\mathcal U}_q(\mathfrak{n}_J)$,
\[
    {\mathcal U}_q(\mathfrak{n}_J) := \langle X_{\beta_1}, \dots X_{\beta_N}
    \rangle \subseteq {\mathcal U}_q (\mathfrak{n}^+),
\]
and we refer to it as the \textit{quantized nilradical of $\mathfrak{p}_J$}, or
\textit{quantized nilradical} for short. The subalgebra of ${\mathcal
U}_q(\mathfrak{g})$ generated by the negative root vectors
\[
    X_{-\beta_1} = F_{i_1}, X_{-\beta_2} = T_{s_{i_1}} F_{i_2},...,
    X_{-\beta_N} = T_{s_{i_1}} \cdots T_{s_{i_{N-1}}} F_{i_N}.
\]
is isomorphic to ${\mathcal U}_q(\mathfrak{n}_J)$.

Quantized nilradicals belong to a larger class of algebras called quantum
Schubert cell algebras, which are indexed by elements $w$ in the Weyl group.
More generally, given a reduced expression of a Weyl group element $w$, the
corresponding quantum Schubert cell algebra ${\mathcal U}_q^+[w]$ can be
constructed in the same way as ${\mathcal U}_q(\mathfrak{n}_J)$ by replacing a
reduced expression for $w_J$ above with a reduced expression for $w$.  De
Concini, Kac, and Procesi \cite[Proposition 2.2]{DKP} proved that the algebra
${\mathcal U}_q^+[w]$ does not depend on the reduced expression for $w$.
Furthermore, every quantum Schubert cell ${\mathcal U}_q^+[w]$ has a PBW basis
\[
    X_{\beta_1}^{m_1}\cdots X_{\beta_N}^{m_N}, \hspace{.4cm} m_1,...,m_N \in
    \mathbb{Z}_{\geq 0},
\]
of standard monomials, and they have presentations as iterated Ore extensions,
\[
    {\mathcal U}_q^+[w] = \mathbb{K}[X_{\beta_1}][X_{\beta_2}; \sigma_2,
    \delta_2] \cdots [X_{\beta_N}; \sigma_N, \delta_N].
\]
For $1 < i < j \leq N$, define the \textit{interval subalgebra}
\begin{equation}
    \label{definition, interval subalgebra}
    {\mathbf U}_{[i, j]} := \langle X_i,X_{i + 1},\dots, X_j \rangle \subseteq
    {\mathcal U}_q^+[w]
\end{equation}
as the subalgebra generated by $X_i, X_{i + 1},\dots, X_j$. Standard monomials
\[
    X_{\beta_i}^{m_i} \cdots X_{\beta_j}^{m_j}, \hspace{5mm} m_i,\dots, m_j \in
    \mathbb{Z}_{\geq 0}
\]
form a basis of $\mathbf{U}_{[i, j]}$.  The Levendorskii-Soibelmann
straightening rule \cite[Prop. 5.5.2]{LS} tells us that for all $1 \leq i < j
\leq N$,
\[
    [X_{\beta_i}, X_{\beta_j}] \in \mathbf{U}_{[i + 1, j - 1]} \cap {\mathcal
    U}_q(\mathfrak{g})_{\beta_i + \beta_j},
\]
(recall (\ref{def, q-commutators})). As a consequence of the straightening
rule, we have the following corollary.
\begin{corollary}

    \label{LS corollary}

    If $1\leq i < j \leq N$ and there fails to exist a nonnegative integral
    combination of roots in $\left\{\beta_{i+1},\dots, \beta_{j-1}\right\}$
    that sum to $\beta_i + \beta_j$, then $[X_{\beta_i}, X_{\beta_j}] = 0$.

\end{corollary}

Furthermore, every quantum Schubert cell algebra ${\mathcal U}_q^+[w]$ has a
quantum cluster algebra structure (provided the deformation parameter $q$
satisfies some minor conditions) \cite{Goodearl Yakimov} with a set of frozen
variables
\[
    \left\{ \Theta_1,\dots, \Theta_r \right\}, \hspace{5mm} r = \# \left\{
    s_i\in W : s_i < w \text{ w.r.t. the Bruhat order} \right\}.
\]
The \textit{normal} subalgebra of ${\mathcal U}_q^+[w]$ (the subalgebra
generated by the normal elements) is generated by the frozen variables
\cite[Proposition 2.7]{GY2}.

\section{Automorphisms of quantum Schubert cells}

In this section, we assume that $R$ is a quantum Schubert cell algebra, say $R
= {\mathcal U}_q^+[w] \subseteq {\mathcal U}_q(\mathfrak{g})$. Fix a reduced
expression
\[
    w = s_{i_1}s_{i_2}\cdots s_{i_N} \in W
\]
and let
\[
    X_{\beta_1}, X_{\beta_2},\dots, X_{\beta_N} \in R
\]
be the corresponding Lusztig root vectors.  Recall that $R$ can be written as
an iterated Ore extension
\begin{equation}
    \label{Ore presentation}
    R = \mathbb{K}[X_{\beta_1}][X_{\beta_2}; \sigma_2, \delta_2]\cdots
    [X_{\beta_N}; \sigma_N, \delta_N].
\end{equation}

Observe that the algebraic torus ${\mathcal H} =
\left(\mathbb{K}^\times\right)^r$ of rank $r =
\operatorname{rank}(\mathfrak{g})$ acts canonically on ${\mathcal
U}_q(\mathfrak{g})$ via algebra automorphisms. An element $h = (h_1,\dots, h_r)
\in {\mathcal H}$ acts by the rule
\[
    h.E_i = h_iE_i, \hspace{5mm}h.F_i = h_i^{-1}F_i, \hspace{5mm} h.K_\mu =
    K_\mu,
\]
for all $1\leq i \leq r$ and $\mu \in Q$. This action is preserved by $R$, and
each Lusztig root vector $X_{\beta_i}$ is an ${\mathcal H}$-eigenvector. In
fact, every $Q$-homogeneous element is an ${\mathcal H}$-eigenvector.
Furthermore, the iterated Ore extension presentation in (\ref{Ore
presentation}) is a symmetric Cauchon-Goodearl-Letzter (CGL) extension
presentation for $R$ (see e.g.  \cite[Theorem 9.1.b]{Goodearl Yakimov}).

\subsection{The function $\mathbf{\eta}$}

Following \cite{GY2}, to every iterated Ore extension presentation $R$, as in
(\ref{Ore presentation}) above, we define the rank of $R$
\[
    \operatorname{rank}(R) = \#\left\{ k \in \left\{1,\dots, N\right\} :
    \delta_k = 0\right\}.
\]
Let $S$ be a set of cardinality $\operatorname{rank}(R)$, and let $\eta :
\left\{1,\dots, N\right\} \to S$ be a function such that $\eta(\left\{k \in
\left\{1,\dots, N\right\} : \delta_k = 0\right\}) = S$. That is, we assign to
each trivial derivation $\delta_k$ a unique element in $S$. We require also
that, for every $k\in \left\{1,\dots, N\right\}$ such that $\delta_k \neq 0$,
\[
    \eta(k) = \eta\left(\operatorname{max}\left\{\ell \in \left\{1,\dots, k -
    1\right\}: \delta_k(X_{\beta_\ell}) \neq 0\right\}\right).
\]
The existence of such a function $\eta$ was proved in \cite[Theorem 4.3]{GY2},
and it plays a key role in determining the homogeneous prime elements for any
CGL extension $R$. When $R = {\mathcal U}_q^+[w]$, the rank of $R$ agrees with
the cardinality of the \textit{support} of $w$,
\[
    \operatorname{supp}(w) := \left\{i\in \mathbf{I} : s_i < w \text{ w.r.t.
    the Bruhat order}\right\}.
\]
In this setting, the function $\eta: \left\{1,\dots,N\right\} \to
\operatorname{supp}(w)$ can be defined by the rule
\begin{equation}
    \label{eta, reduced expr}
    \eta(k) = i_k, \hspace{5mm} (1\leq k \leq N),
\end{equation}
(see e.g. \cite[Theorem 9.5]{Goodearl Yakimov}).

\subsection{The core of $\mathbf{R}$}

Following \cite[Section 4.1]{GY1}, define $P_x(R)$ to be the set of those $i
\in \left\{1,\dots, N\right\}$ such that $X_{\beta_i}$ is prime. By
\cite[Proposition 2.6]{GY1},
\begin{equation}
    \label{definition PxR}
    P_x(R) = \left\{ i \in \left\{1,\dots, N\right\} : \left\{i\right\} =
    \eta^{-1}(\eta(i)) \right\}.
\end{equation}

For $1 \leq j < k \leq N$, the element $Q_{jk} := [X_{\beta_j}, X_{\beta_k}]$
can be written uniquely as a linear combination of monomials $X_{\beta_{j
+1}}^{m_{j+1}} \cdots X_{\beta_{k-1}}^{m_{k-1}}$.  Let $F_x(R)$ be the set of
those $i \in P_x(R)$ such that $X_{\beta_i}$ does not appear in any $Q_{jk}$.
More precisely, no monomial in $Q_{jk}$ with a nonzero coefficient contains a
positive power of $X_{\beta_i}$.  Define $C_x(R) := \left\{1,\dots, N\right\}
\backslash F_x(R)$. The \textit{core} of $R$, denoted by ${\mathcal C}(R)$, is
defined as the subalgebra generated by the $X_{\beta_i}$'s with $i\in C_x(R)$,
\[
    {\mathcal C}(R) := \mathbb{K}\langle X_{\beta_i} : i \in C_x(R) \rangle.
\]

\subsection{Diagonal and graded automorphisms}

An algebra automorphism $\phi: R \to R$ that sends every Lusztig root vector
$X_{\beta_i}$ ($1\leq i \leq N$) to a scalar multiple of itself will be called
a \textit{diagonal} automorphism. This notion is dependent upon the choice of
reduced expression for $w$. Hence, whenever we refer to automorphisms of this
type, we have a fixed reduced expression for $w$ in mind. The set of diagonal
automorphisms is a subgroup of the automorphism group $\operatorname{Aut}(R)$
of $R$. We will denote this subgroup by $\operatorname{Diag-Aut}(R)$. Thus, for
an algebra automorphism $\phi: R \to R$,
\[
    \phi \in \operatorname{Diag-Aut}(R) \iff \phi(X_{\beta_i}) \in
    \mathbb{K}^\times X_{\beta_i}, \hspace{5mm}(1\leq i \leq N).
\]
From Section \ref{Z grading section}, recall that every coweight $\lambda \in
P^\vee$ induces a $\mathbb{Z}$-grading on ${\mathcal U}_q(\mathfrak{g})$. With
this, the subalgebra $R = {\mathcal U}_q^+[w] \subseteq {\mathcal
U}_q(\mathfrak{g})$ inherits this grading,
\[
    R = \bigoplus_{d\in\mathbb{Z}} R_d, \hspace{5mm} R_d := \left\{ u \in R
    \cap {\mathcal U}_q(\mathfrak{g})_\mu : \langle \mu, \lambda \rangle = d
    \right\}.
\]
We assume throughout that $\lambda \in P^\vee$ is chosen so that the induced
grading satisfies the following conditions:

\begin{enumerate}

    \item $R = R_0 \oplus R_1 \oplus R_2 \oplus \cdots$ (that is to say, $R_d =
        0$ whenever $d < 0$),

    \item $R_d$ is finite dimensional for every $d\geq 0$ (i.e. $R$ is locally
        finite),

    \item $R_0 = \mathbb{K}$ (i.e. $R$ is connected), and

    \item $R$ is generated, as an algebra, by $R_1$.

\end{enumerate}

These conditions mimic the standard grading on a commutative polynomial
ring $\mathbb{K}[z_1,\dots, z_N]$, where each variable $z_i$ is assigned degree
$1$.  It is always possible to choose $\lambda$ so that the first three
conditions above are satisfied. For example, $\lambda = \sum_{i\in
\mathbf{I}}\varpi_i^\vee$ is one such choice. However, it is not always
possible to select $\lambda$ such that all four conditions are met. To give
one example, it is not too difficult to verify that such a $\lambda$ fails to
exist for the case when the underlying Lie algebra is of type $G_2$ and $w =
s_2s_1s_2$.

An algebra automorphism $\phi: R \to R$ is a \textit{graded} algebra
automorphism if it respects the $\mathbb{Z}_{\geq 0}$-grading. That is to say,
\[
    \phi(R_d) = R_d
\]
for all $d\geq 0$. The set of graded automorphisms is a subgroup of the
automorphism group of $R$. We denote the subgroup of graded automorphisms by
$\operatorname{Gr-Aut}(R)$. Observe we have a chain of subgroups,
\[
    \operatorname{Diag-Aut}(R) \subseteq \operatorname{Gr-Aut}(R) \subseteq
    \operatorname{Aut}(R).
\]
Using (\ref{eta, reduced expr}) and (\ref{definition PxR}), one can easily
determine the set $P_x(R)$ from the reduced expression $w = s_1 \cdots s_N$.
In many cases $P_x(R)$ is empty (and ${\mathcal C}(R) = R$), and in such
situations we have the most control over the automorphisms of $R$ (see e.g.
\cite{GY1}). The following theorem describes sufficient conditions on $R$ to
conclude that every automorphism of $R$ is graded.

\begin{theorem}

    \label{when Aut is Gr-Aut}

    Suppose $R = {\mathcal U}_q^+[w]$ is a quantum Schubert cell algebra with
    Lusztig root vectors $X_{\beta_1},\dots, X_{\beta_N}$. Suppose ${\mathcal
    C}(R) = R$. Suppose also that $R$ is connected graded, locally finite, and
    generated by $R_1$. For every radical root $\beta_i \in \Delta_w$ with
    $X_{\beta_i} \in R_1$, suppose there exists $\beta_j\in \Delta_w$ such that
    $X_{\beta_i}X_{\beta_j} = \kappa X_{\beta_j}X_{\beta_i}$ for some scalar
    $\kappa \neq 1$. Then every algebra automorphism of $R$ is a graded
    automorphism. In other words,
    \[
        \operatorname{Aut}(R) = \operatorname{Gr-Aut}(R).
    \]

\end{theorem}

\begin{proof}

    It was shown in \cite[Theorem 4.2]{GY1} that if $R$ is a symmetric
    saturated CGL extension which is a connected graded algebra, then every
    unipotent automorphism restricted to ${\mathcal C}(R)$ is the identity.
    Since ${\mathcal C}(R) = R$, then the identity is the only unipotent
    automorphism of $R$. As a consequence of \cite[Lemma 4.7]{GY1} every
    automorphism $\phi$ is graded provided $\phi(R_d) \subseteq \oplus_{j\geq
    d} R_j$ for all $d \geq 0$. However, this condition was established in
    \cite[Proposition 4.2]{LL}.

\end{proof}

\subsection{The normal subalgebra $\mathbf{{\mathcal N}(R)}$ and the sets
$\mathbf{C_d^m}$ and $\mathbf{\gamma_{d, \ell}^m}$}

Following \cite{GY2}, let ${\mathcal N}(R)$ be the \textit{normal subalgebra}
of $R$. It is the subalgebra generated by the normal elements of $R$. By
\cite[Theorem 4.3]{GY2}, ${\mathcal N}(R)$ is a generated by a
finite set of $Q$-homogeneous prime elements
\[
    \left\{\Theta_i : i \in \operatorname{supp}(w)\right\} \subseteq R.
\]
We remark here that the element $\Theta_i \in R$ is written as
$\Delta_{\varpi_i,w\varpi_i}$ in \cite[Section 9.4]{Goodearl Yakimov}.  We have
the following commutation relations,
\begin{equation}
    \label{commutation with normal subalgebra}
    u\Theta_i  = q^{-\langle \left(1 + w\right)\varpi_i, \mu \rangle}
    \Theta_iu, \hspace{5mm} \operatorname{deg}_Q(u) = \mu, \hspace{3mm} i
    \in \operatorname{supp}(w),
\end{equation}
(see e.g. \cite[Eq. 3.30]{Y2} and \cite[Eq. 9.23]{Goodearl Yakimov}).

For $d\in \mathbb{Z}_{\geq 0}$, $m \in \mathbb{Z}$, and $x \in R_1$, define
\begin{equation}
    \label{definition C}
    \begin{split}
        &C_d^m := \left\{ x \in R_1 : xy = q^m yx
        \text{ for all } y \in {\mathcal N}(R)_d
        \right\},
        \\
        &V_d^m(x) := \left\{ y \in {\mathcal N}\left(R\right)_d : xy = q^myx
        \right\}
    \end{split}
\end{equation}

For $d,\ell \in \mathbb{Z}_{\geq 0}$ and $m \in \mathbb{Z}$, define the set
\begin{equation}
    \label{definition gamma}
    \gamma_{d,\ell}^m := \left\{ x \in R_1 :
    \operatorname{dim}_{\mathbb{K}} \left(  V_d^m(x)
    \right) = \ell \right\}.
\end{equation}

The following proposition is an easy observation.

\begin{proposition}
    \label{C gamma, fixed}

    Suppose $\phi \in \operatorname{Gr-Aut}(R)$. Let $C_d^m$ and $\gamma_{d,
    \ell}^m$ be as defined in (\ref{definition C}) and (\ref{definition
    gamma}).  Then each $C_d^m$ is a $\phi$-invariant subspace of $R_1$, and
    each $\gamma_{d,\ell}^m$ is a $\phi$-invariant subset of $R_1$.

\end{proposition}

The following theorem gives us sufficient conditions to determine when a
standard generator $\Theta_j$ of ${\mathcal N}(R)$ gets sent to a nonzero
scalar multiple of itself by a graded algebra automorphism $\phi: R \to R$.

\begin{theorem}

    \label{normal gen, fixed}

    Suppose $\phi\in \operatorname{Gr-Aut}(R)$. For every $i\in
    \operatorname{supp}(w)$, let $d_i$ be the degree of $\Theta_i$. That is,
    $\Theta_i \in {\mathcal N}(R)_{d_i}$. If, for some $j\in
    \operatorname{supp}(w)$, there fails to exist a nonnegative integral
    combination of numbers in $\left\{ d_i : i \in \operatorname{supp}(w)
    \text{ and } i \neq j\right\}$ that sum to $d_j$, then
    \[
        \phi (\Theta_j) \in \mathbb{K}^\times \Theta_j.
    \]

\end{theorem}

\begin{proof}

    Ordered monomials in the $\Theta_i$'s form a basis (over $\mathbb{K}$) of
    the normal subalgebra ${\mathcal N}(R)$ \cite[Theorem 4.6]{GY2}. As
    there fails to exist a nonnegative integral combination of numbers in
    $\left\{ d_i : i \in \operatorname{supp}(w) \text{ and } i \neq j\right\}$
    that sum to $d_j$, this implies that ${\mathcal N}(R)_{d_j}$ is a
    one-dimensional vector space over $\mathbb{K}$ spanned by the element
    $\Theta_j$. Since ${\mathcal N}(R)$ is an invariant subalgebra of $R$ under
    $\phi$, then we have $\phi ({\mathcal N}(R)_{d_j}) = {\mathcal
    N}(R)_{d_j}$. Hence $\phi(\Theta_j) \in \mathbb{K}^\times \Theta_j$.

\end{proof}

The following theorem gives sufficient conditions to determine when a Lusztig
root vector $X_\beta$ gets sent to a nonzero scalar multiple of itself by a
graded algebra automorphism $\phi: R \to R$.

\begin{theorem}

    \label{root vector, fixed}

    Suppose $\phi \in \operatorname{Gr-Aut}(R)$ and $\phi(\Theta_i) \in
    \mathbb{K}^\times \Theta_i$ for some $i\in \operatorname{supp}(w)$.
    Suppose also that there exists a radical root $\beta \in \Delta_w$ with
    $X_\beta \in R_1$ such that
    \[
        \langle \beta - \beta^\prime, \left(1 + w \right) \varpi_i \rangle \neq
        0
    \]
    for every radical root $\beta^\prime \in \Delta_w \backslash
    \left\{\beta\right\}$ with $X_{\beta^\prime} \in R_1$, then
    \[
        \phi(X_\beta) \in \mathbb{K}^\times X_\beta.
    \]

\end{theorem}

\begin{proof}

    Suppose $x_1,\dots,x_n$ is a list of the Lusztig root vectors in $R_1$.
    Without loss of generality, assume $x_1 = X_\beta$.  There are integers
    $d_1,\dots,d_n$ which can be computed explicitly using (\ref{commutation
    with normal subalgebra}) such that $x_j\Theta_i = q^{d_j} \Theta_i x_j$.
    The given hypotheses imply that $d_1$ is not equal to any number in
    $\left\{d_2, \dots, d_n\right\}$.

    The automorphism $\phi$ sends $x_1$ to a linear combination of $x_1,\dots,
    x_n$, say $\phi(x_1) = \sum c_jx_j$, ($c_j \in \mathbb{K}$). Applying
    $\phi$ to the relation $x_1\Theta_i = q^{d_1}\Theta_i x_1$ yields $\sum
    c_jx_j\Theta_i = \sum c_jq^{d_1}\Theta_i x_j = \sum c_j q^{d_j} \Theta_i
    x_j$.  Thus, $\sum c_j (q^{d_1} - q^{d_j}) \Theta_i x_j = 0$. The elements
    $\Theta_i x_1, \dots, \Theta_i x_n$ are $Q$-homogeneous and have distinct
    degrees with respect to the $Q$-gradation. Hence, each of the coefficients
    $c_j(q^{d_1} - q^{d_j})$ equals zero. Since $q$ is not a root of unity,
    $c_2 = \cdots = c_n = 0$.

\end{proof}

Using the same techniques in the proof of Theorem \ref{root vector, fixed}
above, we have a more general result.

\begin{theorem}

    \label{root vector, fixed, 2}

    Suppose $\phi \in \operatorname{Gr-Aut}(R)$. Define the set
    \[
        S(\phi, w) := \left\{ i \in \operatorname{supp}(w) \mid \phi(\Theta_i)
        \in \mathbb{K}^\times \Theta_i \right\}.
    \]
    Suppose that $\beta \in \Delta_w$ is a radical root with $X_\beta \in R_1$
    and satisfies the condition that for every radical root $\beta^\prime \in
    \Delta_w \backslash \left\{\beta\right\}$ with $X_{\beta^\prime} \in R_1$,
    there exists $i \in S(\phi, w)$ such that $\langle \beta - \beta^\prime,
    \left(1 + w \right) \varpi_i \rangle \neq 0$, then
    \[
        \phi(X_\beta) \in \mathbb{K}^\times X_\beta.
    \]

\end{theorem}

By applying the theorems above, we can prove that Conjecture \ref{conj} holds,
for example, when the underlying Lie algebra $\mathfrak{g}$ is of type $G_2$.

\begin{theorem}
    \label{Theorem, G2}

    If $\mathfrak{g}$ is the Lie algebra of type $G_2$ and $J$ is a nonempty
    subset of simple roots of $\mathfrak{g}$, then
    $\operatorname{Aut}({\mathcal U}_q(\mathfrak{n}_J)) \cong
    \left(\mathbb{K}^\times\right)^2$.

\end{theorem}

\begin{proof}

    We consider the reduced expression $w_o = s_1s_2s_1s_2s_1s_2$ for the
    longest element of the Weyl group of $\mathfrak{g}$. The corresponding
    radical roots and root vectors associated to this reduced expression will
    be denoted by $\beta_1,\dots,\beta_6$ and $x_1,\dots,x_6$, respectively.
    In other words, we define $x_i := X_{\beta_i}$ ($i=1,\dots,6$).  The
    positive part ${\mathcal U}_q^+(\mathfrak{g})$ of ${\mathcal
    U}_q(\mathfrak{g})$ is generated by $x_1,\dots, x_6$. The defining
    relations appear in the work of Hu and Wang \cite{HW}, where they index the
    root vectors by Lyndon words. The correspondence between our notation and
    their notation is $x_1 \leftrightarrow E_1$, $[3]_q!x_2 \leftrightarrow
    E_{1112}$, $[2]_qx_3 \leftrightarrow E_{112}$, $[3]_q!x_4 \leftrightarrow
    E_{11212}$, $x_5 \leftrightarrow E_{12}$, $x_6 \leftrightarrow E_2$.  From
    \cite[Eqns. 2.2 - 2.7 and Lemma 3.1]{HW}, the defining relations in
    ${\mathcal U}_q^+(\mathfrak{g})$ are

    \begin{enumerate}

            \begin{multicols}{2}

            \item[] $x_1 x_2 = q^3 x_2 x_1$,

            \item[] $x_1 x_3 = q x_3 x_1 + [3]_q x_2$,

            \item[] $x_1 x_4 = x_4x_1 + q\widehat{q} x_3^2$,

            \item[] $x_1 x_5 = q^{-1} x_5 x_1 + [2]_q x_3$,

            \item[] $x_1 x_6 = q^{-3} x_6 x_1 + x_5$,

            \item[] $x_2 x_3 = q^3 x_3 x_2$,

            \item[] $x_2 x_4 = q^3 x_4 x_2 + \eta x_3^3$,

            \item[] $x_2 x_5 = x_5 x_2 + q\widehat{q} x_3^2$,

            \item[] $x_2 x_6 = q^{-3} x_6 x_2 + \widehat{q} x_3
                x_5 + \zeta x_4$,

            \item[] $x_3 x_4 = q^3 x_4 x_3$,

            \item[] $x_3 x_5 = q x_5 x_3 + [3]_q x_4$,

            \item[] $x_3 x_6 = x_6 x_3 + q\widehat{q} x_5^2$,

            \item[] $x_4 x_5 = q^3 x_5 x_4$,

            \item[] $x_4 x_6 = q^3 x_6 x_4 + \eta x_5^3$

            \item[] $x_5 x_6 = q^3 x_6 x_5$,
        \end{multicols}
    \end{enumerate}
    \noindent where $\zeta := q^{-3} - q^{-1} - q\in
    \mathbb{K}$ and $\eta = q^3 \frac{\widehat{q}^2}{[3]_q} \in
    \mathbb{K}$.

    Observe that when $J$ is a singleton, the parabolic element $w_o^Jw_o$ has
    a unique reduced expression, and it appears as a substring of the reduced
    expression for the longest element $w_o$. Hence, each quantized nilradical
    ${\mathcal U}_q(\mathfrak{n}_J)$ is isomorphic to an \textit{interval}
    subalgebra of ${\mathcal U}_q^+(\mathfrak{g})$. In particular, when $J =
    \left\{\alpha_1\right\}$, ${\mathcal U}_q(\mathfrak{n}_J)$ is isomorphic to
    the subalgebra generated by $x_1,\dots,x_5$, whereas if $J =
    \left\{\alpha_2\right\}$, the corresponding quantized nilradical is
    isomorphic to the subalgebra generated by $x_2,\dots,x_6$.  For
    convenience, we will identify each quantized nilradical with an appropriate
    interval subalgebra.

    First consider the case when $J = \left\{ \alpha_1 \right\}$. Here we
    choose the coweight $\lambda = \varpi_1 \in P^\vee$ to equip ${\mathcal
    U}_q(\mathfrak{n}_J)$ with a $\mathbb{N}$-gradation. With respect to this
    grading, the degree one generators are $x_1$ and $x_5$. The defining
    relations verify that ${\mathcal U}_q(\mathfrak{n}_J)$ is generated as an
    algebra by its degree one elements. By using (\ref{definition PxR}), we
    have ${\mathcal C}\left( {\mathcal U}_q \left(\mathfrak{n}_J\right)\right)$
    is generated by $x_1$, $x_3$, and $x_5$. Thus, ${\mathcal C}\left({\mathcal
    U}_q(\mathfrak{n}_J)\right) = {\mathcal U}_q(\mathfrak{n}_J)$ and Theorem
    \ref{when Aut is Gr-Aut} be can applied to conclude that every automorphism
    of ${\mathcal U}_q(\mathfrak{n}_J)$ preserves the $\mathbb{N}$-grading. The
    elements $\Theta_1$ and $\Theta_2$ of the normal subalgebra have degrees
    $4$ and $6$, respectively. Hence, by Theorem \ref{normal gen, fixed}, every
    automorphism sends $\Theta_1$ and $\Theta_2$ to nonzero multiples of
    themselves. Finally, one can verify that Theorem \ref{root vector, fixed,
    2} can be applied to conclude that every automorphism sends $x_1$ and $x_5$
    to nonzero multiples of themselves. Therefore, every automorphism of
    ${\mathcal U}_q(\mathfrak{n}_J)$ is a diagonal automorphism. As a CGL
    extension, the algebra ${\mathcal U}_q(\mathfrak{n}_J)$ has rank 2. Thus,
    by \cite[Theorems 5.3 and 5.5]{GY2}, $\operatorname{Aut}({\mathcal
    U}_q(\mathfrak{n}_J)) \cong \left(\mathbb{K}^\times\right)^2$.

    The case when $J = \left\{\alpha_2 \right\}$ is treated similarly, except
    now we choose the coweight $\varpi_2 \in P^\vee$ to equip the corresponding
    quantized nilradical with a $\mathbb{N}$-gradation. Here, we identify
    ${\mathcal U}_q(\mathfrak{n}_J)$ with the subalgebra of ${\mathcal
    U}_q^+(\mathfrak{g})$ generated by $x_2,\dots,x_6$. With this
    identification, the degree one root vectors of ${\mathcal
    U}_q(\mathfrak{n}_J)$ are $x_2$, $x_3$, $x_5$, and $x_6$. Again, we see
    that ${\mathcal U}_q(\mathfrak{n}_J)$ is generated as an algebra by its
    elements of degree one, and the core ${\mathcal C}({\mathcal
    U}_q(\mathfrak{n}_J))$ coincides with ${\mathcal U}_q(\mathfrak{n}_J)$. All
    of the hypotheses of Theorem \ref{when Aut is Gr-Aut} apply. Hence, every
    automorphism preserves the $\mathbb{N}$-grading. The elements $\Theta_1$
    and $\Theta_2$ of the normal subalgebra have degrees $2$ and $4$,
    respectively, in this setting.  Thus, by Theorem \ref{normal gen, fixed},
    every automorphism of ${\mathcal U}_q(\mathfrak{n}_J)$ sends $\Theta_1$ to
    a nonzero multiple of itself. Finally, by applying Theorem \ref{root
    vector, fixed}, we can conclude that every automorphism sends each degree
    one root vector to a nonzero multiple of itself. Thus, every automorphism
    is a diagonal automorphism.  Finally, \cite[Theorems 5.3 and 5.5]{GY2}
    imply that $\operatorname{Aut}({\mathcal U}_q(\mathfrak{n}_J)) \cong
    \left(\mathbb{K}^\times\right)^2$.

    The case when $J = \left\{ \alpha_1, \alpha_2 \right\}$ has already been
    established in \cite{Y3}.

\end{proof}

Theorems \ref{root vector, fixed} and \ref{root vector, fixed, 2} can be
applied in several other cases to conclude that every graded automorphism of a
quantized nilradical sends each degree one generator to a multiple of itself.
If it can also be established that some particular quantized nilradical
${\mathcal U}_q(\mathfrak{n}_J)$ satisfies all of the hypotheses of Theorem
\ref{when Aut is Gr-Aut}, then the results of Goodearl and Yakimov in
\cite[Theorems 5.3 and 5.5]{GY2} can be applied to conclude that
$\operatorname{Aut}({\mathcal U}_q(\mathfrak{n}_J)) \cong
\left(\mathbb{K}^\times\right)^{\operatorname{rank}(\mathfrak{g})}$.

\begin{proposition}
    \label{prop}

    As above, let $\mathfrak{g}$ be a finite dimensional complex simple Lie
    algebra, and let $J$ be a nonempty set of simple roots. Choose the coweight
    $\lambda := \sum_{j\in J} \varpi_j \in P^\vee$ to equip the quantized
    nilradical ${\mathcal U}_q(\mathfrak{n}_J)$ with a $\mathbb{N}$-gradation.
    The following list is an exhaustive list of all cases with
    $\operatorname{rank}(\mathfrak{g}) \leq 9$ such that Theorems \ref{root
    vector, fixed} and \ref{root vector, fixed, 2} can be applied to conclude
    that every graded automorphism of ${\mathcal U}_q(\mathfrak{n}_J)$ sends
    each degree one Lusztig root vector $X_\beta$ to a multiple of itself.

    \begin{enumerate}

        \item $\mathfrak{g}$ is of type $B_6$ and $J$ is $\left\{\alpha_2,
            \alpha_5\right\}$ or $\left\{\alpha_2, \alpha_4, \alpha_5\right\}$.

        \item $\mathfrak{g}$ is of type $B_7$ or $B_8$ and $J =
            \left\{\alpha_2, \alpha_6, \alpha_7\right\}$.

        \item $\mathfrak{g}$ is of type $C_5$ and $J$ is one of
            $\left\{\alpha_3, \alpha_5\right\}$,
            $\left\{\alpha_1, \alpha_3, \alpha_5\right\}$,
            $\left\{\alpha_2, \alpha_3, \alpha_5\right\}$,

            $\left\{\alpha_2, \alpha_4, \alpha_5\right\}$.

        \item $\mathfrak{g}$ is of type $C_6$ and $J$ is $\left\{\alpha_2,
            \alpha_4, \alpha_6\right\}$ or $\left\{\alpha_1, \alpha_3,
            \alpha_5, \alpha_6\right\}$.

        \item $\mathfrak{g}$ is of type $D_7$ and $J$ is $\left\{\alpha_2,
            \alpha_4, \alpha_6\right\}$ or $\left\{\alpha_2, \alpha_4,
            \alpha_7\right\}$.

        \item $\mathfrak{g}$ is of type $D_8$ and $J$ is $\left\{\alpha_2,
            \alpha_4, \alpha_6, \alpha_7\right\}$ or $\left\{\alpha_2,
            \alpha_4, \alpha_6, \alpha_8\right\}$.

        \item $\mathfrak{g}$ is of type $E_7$ and $J$ is one of
            $\left\{\alpha_3, \alpha_5\right\}$,
            $\left\{\alpha_4, \alpha_5\right\}$,
            $\left\{\alpha_4, \alpha_7\right\}$,
            $\left\{\alpha_1, \alpha_4, \alpha_5\right\}$,

            $\left\{\alpha_1, \alpha_4, \alpha_7\right\}$,
            $\left\{\alpha_2, \alpha_3, \alpha_6\right\}$,
            $\left\{\alpha_2, \alpha_4, \alpha_6\right\}$,
            $\left\{\alpha_3, \alpha_4, \alpha_5\right\}$,
            $\left\{\alpha_3, \alpha_4, \alpha_7\right\}$,

            $\left\{\alpha_3, \alpha_5, \alpha_6\right\}$,
            $\left\{\alpha_4, \alpha_5, \alpha_6\right\}$,
            $\left\{\alpha_4, \alpha_6, \alpha_7\right\}$,
            $\left\{\alpha_1, \alpha_2, \alpha_4, \alpha_6\right\}$,

            $\left\{\alpha_1, \alpha_4, \alpha_5, \alpha_6\right\}$,
            $\left\{\alpha_1, \alpha_4, \alpha_6, \alpha_7\right\}$,
            $\left\{\alpha_2, \alpha_3, \alpha_4, \alpha_6\right\}$,
            $\left\{\alpha_2, \alpha_3, \alpha_5, \alpha_7\right\}$,

            $\left\{\alpha_3, \alpha_4, \alpha_5, \alpha_6\right\}$,
            $\left\{\alpha_3, \alpha_4, \alpha_6, \alpha_7\right\}$.

        \item $\mathfrak{g}$ is of type $E_8$ and $J$ is one of
            $\left\{\alpha_4, \alpha_6\right\}$,
            $\left\{\alpha_4, \alpha_7\right\}$,
            $\left\{\alpha_4, \alpha_8\right\}$,
            $\left\{\alpha_1, \alpha_4, \alpha_6\right\}$,

            $\left\{\alpha_1, \alpha_4, \alpha_7\right\}$,
            $\left\{\alpha_1, \alpha_4, \alpha_8\right\}$,
            $\left\{\alpha_2, \alpha_3, \alpha_7\right\}$,
            $\left\{\alpha_3, \alpha_4, \alpha_6\right\}$,
            $\left\{\alpha_3, \alpha_4, \alpha_7\right\}$,

            $\left\{\alpha_3, \alpha_4, \alpha_8\right\}$,
            $\left\{\alpha_3, \alpha_5, \alpha_7\right\}$,
            $\left\{\alpha_4, \alpha_5, \alpha_7\right\}$,
            $\left\{\alpha_4, \alpha_6, \alpha_7\right\}$,
            $\left\{\alpha_4, \alpha_6, \alpha_8\right\}$,

            $\left\{\alpha_4, \alpha_7, \alpha_8\right\}$,
            $\left\{\alpha_1, \alpha_2, \alpha_4, \alpha_6\right\}$,
            $\left\{\alpha_1, \alpha_2, \alpha_4, \alpha_8\right\}$,
            $\left\{\alpha_1, \alpha_4, \alpha_5, \alpha_7\right\}$,

            $\left\{\alpha_1, \alpha_4, \alpha_6, \alpha_7\right\}$,
            $\left\{\alpha_1, \alpha_4, \alpha_6, \alpha_8\right\}$,
            $\left\{\alpha_2, \alpha_3, \alpha_5, \alpha_7\right\}$,
            $\left\{\alpha_2, \alpha_4, \alpha_6, \alpha_8\right\}$,

            $\left\{\alpha_3, \alpha_4, \alpha_5, \alpha_7\right\}$,
            $\left\{\alpha_3, \alpha_4, \alpha_6, \alpha_7\right\}$,
            $\left\{\alpha_3, \alpha_4, \alpha_6, \alpha_8\right\}$,
            $\left\{\alpha_1, \alpha_2, \alpha_4, \alpha_6, \alpha_7\right\}$,

            $\left\{\alpha_1, \alpha_2, \alpha_4, \alpha_6, \alpha_8\right\}$,
            $\left\{\alpha_2, \alpha_3, \alpha_4, \alpha_6, \alpha_8\right\}$.

        \item $\mathfrak{g}$ is of type $F_4$ and $J$ is one of
            $\left\{\alpha_3\right\}$,
            $\left\{\alpha_1, \alpha_3\right\}$,
            $\left\{\alpha_2, \alpha_3\right\}$,
            $\left\{\alpha_2, \alpha_4\right\}$.

        \item $\mathfrak{g}$ is of type $G_2$ and $J = \left\{\alpha_1\right\}$
            or $J = \left\{\alpha_2\right\}$

    \end{enumerate}

\end{proposition}

We choose the first case listed above, namely when $\mathfrak{g}$ is of type
$B_6$ and $J = \left\{ \alpha_2, \alpha_5 \right\}$, and completely determine
its automorphism group. It remains to show that every automorphism preserves
the $\mathbb{N}$-grading in this case.

\begin{theorem}

    \label{Theorem, B6}

    If $\mathfrak{g}$ is of type $B_6$ and $J = \left\{\alpha_2,
    \alpha_5\right\}$, then $\operatorname{Aut}({\mathcal U}_q(\mathfrak{n}_J))
    \cong \left(\mathbb{K}^\times\right)^6$.

\end{theorem}

\begin{proof}

    Consider the reduced expression
    \[
        (s_5 s_6) (s_4 s_5 s_6) (s_3 s_4 s_5 s_6) (s_2 s_3 s_4 s_5 s_6) (s_1
        s_2 s_3 s_4 s_5 s_6) (s_1 s_2 s_3 s_4 s_5) (s_3 s_4) (s_2 s_3) (s_1
        s_2)
    \]
    for the parabolic element $w_o^Jw_o$. The length of this Weyl group element
    is $31$. With this, let $\beta_1,\dots, \beta_{31}$ and $x_1,\dots, x_{31}$
    denote the corresponding radical roots and root vectors, respectively. We
    will choose the coweight $\varpi_2 + \varpi_5$ to equip the quantized
    nilradical ${\mathcal U}_q(\mathfrak{n}_J)$ with a $\mathbb{N}$-gradation.
    By applying (\ref{definition PxR}), we obtain that the core ${\mathcal
    C}({\mathcal U}_q(\mathfrak{n}_J))$ coincides with ${\mathcal
    U}_q(\mathfrak{n}_J)$.  Secondly, for every root vector $x_i$ ($i =
    1,\dots,31$), there is another root vector $x_j$ ($j \neq i$) such that
    $x_i x_j = \kappa x_j x_i$ for some scalar $\kappa \neq 1$. One can always
    simply let $j = i + 1$ or $j = i - 1$. Observe that in the reduced word
    above, every letter $s_k$ is always adjacent to either $s_{k-1}$ or
    $s_{k+1}$. If the $i$-th letter is $s_k$ and the $(i+1)$-th letter is $s_{k
    \pm 1}$, then $x_ix_{i + 1} = q^{\langle\alpha_k, s_k(\alpha_{k \pm
    1})\rangle} x_{i + 1}x_i$. On the other hand, if the $i$-th letter is $s_k$
    and the $(i-1)$-th letter is $s_{k \pm 1}$, then $x_{i-1}x_i =
    q^{\langle\alpha_{k \pm 1}, s_{k \pm 1}(\alpha_k)\rangle} x_ix_{i-1}$. In
    either case, the relevant hypothesis of Theorem \ref{when Aut is Gr-Aut} is
    satisfied.

    Among the $31$ Lusztig root vectors exactly 15 of them are of degree one.
    We leave it to the reader to verify that each remaining root vector $x_i$
    can be written, up to a scalar multiple, as a $q$-commutator $x_jx_k -
    q^{\langle \beta_j, \beta_k\rangle} x_kx_j$ for some $j,k$. Thus ${\mathcal
    U}_q(\mathfrak{n})$ is generated by its elements of degree one. Hence all
    of the hypotheses of Theorem \ref{when Aut is Gr-Aut} are satisfied.
    Therefore every automorphism of ${\mathcal U}_q(\mathfrak{n}_J)$ preserves
    the $\mathbb{N}$-grading.

    The degrees of the normal elements $\Theta_1, \dots, \Theta_6$ are $4, 8,
    10, 12, 14$, and $7$, respectively. Hence, Theorem \ref{normal gen, fixed}
    tells us that every automorphism of ${\mathcal U}_q(\mathfrak{n}_J)$ sends
    $\Theta_1$, $\Theta_3$, and $\Theta_6$ to multiples of themselves. Next,
    Theorem \ref{root vector, fixed, 2} can be applied to conclude that every
    automorphism sends each degree one generator $x_i$ to a multiple of itself.
    Hence, every automorphism is a diagonal automorphism. Finally,
    \cite[Theorems 5.3 and 5.5]{GY2} imply that $\operatorname{Aut}({\mathcal
    U}_q(\mathfrak{n}_J)) \cong \left(\mathbb{K}^\times\right)^6$.

\end{proof}

\section{The automorphism group of ${\mathcal U}_q(\mathfrak{n}_J)$ for
$\mathfrak{g} = F_4$}

We now consider the case when $\mathfrak{g}$ is the Lie algebra of type $F_4$.
In this section we prove that, for every nonempty subset $J$ of simple roots of
$\mathfrak{g}$, the automorphism group of the quantized nilradical ${\mathcal
U}_q(\mathfrak{n}_J)$ is isomorphic to the algebraic torus
$\left(\mathbb{K}^\times\right)^4$.

We consider the following reduced expressions for the longest element $w_o$
of the Weyl group of $\mathfrak{g}$:
\begin{align}
    \label{two reduced expressions}
    \mathbf{R}[w_o] &=
    (1, 2, 1, 3, 2, 3, 1, 2, 4, 3, 2, 1, 3, 2, 3, 4, 3, 2, 3, 1, 2, 3, 4, 3),
    \\
    \mathbf{R}^\prime[w_o] &=
    (4, 1, 2, 3, 4, 2, 1, 3, 2, 3, 1, 2, 4, 3, 2, 1, 3, 2, 3, 4, 3, 2, 3, 2).
\end{align}

The main reason we consider these two particular reduced expressions for $w_o$
is that for every nonempty subset $J\subseteq \Pi$, a reduced expression for
the corresponding parabolic element $w_J$ appears either as a substring of
$\mathbf{R}[w_o]$ or as a substring of $\mathbf{R}^\prime[w_o]$.  Thus, if we
treat ${\mathcal U}_q(\mathfrak{n}^+)$ as an iterated Ore extension over
$\mathbb{K}$, then it follows that each quantized nilradical ${\mathcal
U}_q(\mathfrak{n}_J)$ can be viewed as an \textit{interval subalgebra} (defined
in (\ref{definition, interval subalgebra})) of ${\mathcal
U}_q(\mathfrak{n}^+)$.  This is advantageous for finding explicit presentations
of ${\mathcal U}_q(\mathfrak{n}_J)$ more efficiently.  To make this more
explicit, we will denote the Lusztig root vectors corresponding to the reduced
expressions, $\mathbf{R}[w_o]$ and $\mathbf{R}^\prime[w_o]$, by
\begin{equation}
    \label{xi, yi definition}
    x_i := X_{\beta_i}, \hspace{5mm} y_i := X_{ \beta_i^\prime}, \hspace{5mm}
    (1 \leq i \leq \ell (w_o) =24),
\end{equation}
respectively. Since every parabolic element $w_J$ appears as a substring of
$\mathbf{R}[w_o]$ or $\mathbf{R}^\prime[w_o]$, then every quantized nilradical
${\mathcal U}_q(\mathfrak{n}_J)$ is isomorphic, as an algebra, to a subalgebra
of ${\mathcal U}_q(\mathfrak{n}^+)$ generated by either a contiguous sequence
of $x_i$'s or $y_i$'s. The following example illustrates this when $J =
\left\{\alpha_2, \alpha_4\right\}$.

\begin{example}

    Suppose $J = \left\{\alpha_2, \alpha_4\right\}$. The parabolic element $w_J
    \in W$ has a reduced expression
    \[
        w_J = s_2 s_1 s_3 s_2 s_3 s_1 s_2 s_4 s_3 s_2 s_1 s_3 s_2 s_3 s_4 s_3
        s_2 s_3 s_1 s_2 s_3 s_4.
    \]
    This expression is obtained by removing the first and last letters from
    $\mathbf{R}[w_o]$. Thus ${\mathcal U}_q(\mathfrak{n}_J)$ is isomorphic to
    the interval subalgebra $\mathbf{U}_{[2, 23]} \subseteq {\mathcal
    U}_q(\mathfrak{n}^+)$,
    \[
        {\mathcal U}_q(\mathfrak{n}_J) \cong \mathbb{K}\langle x_2, \dots,
        x_{23} \rangle \subseteq {\mathcal U}_q(\mathfrak{n}^+).
    \]
    With this identification, $x_2, x_3, x_4, x_6, x_7, x_8, x_{22}, x_{23}$ is
    a list of Lusztig root vectors of degree $1$.

\end{example}

\begin{theorem}

    \label{F4: diagonal}

    If $\mathfrak{g}$ is the Lie algebra of type $F_4$ and $J$ is a nonempty
    subset of simple roots of $\mathfrak{g}$, then every automorphism of the
    quantized nilradical ${\mathcal U}_q(\mathfrak{n}_J)$ is a diagonal
    automorphism.

\end{theorem}

\begin{proof}

    The case when $J$ is the full set of simple roots was handled in
    \cite[Theorem 5.1]{Y3}. Thus, we suppose $J$ is a nonempty proper subset of
    the set of simple roots.  Throughout this proof we will let $\phi$ be an
    arbitrary algebra automorphism of ${\mathcal U}_q(\mathfrak{n}_J)$.  The
    algebra ${\mathcal U}_q(\mathfrak{n}_J)$ satisfies the hypotheses of
    Theorem \ref{when Aut is Gr-Aut}. Hence, $\phi$ is a graded automorphism.
    Thus, the sets $C_d^m$ and $\gamma_{d, \ell}^m$ are $\phi$-invariant.

    Our objective is to prove that $\phi$ is a diagonal automorphism. The
    commutation relations given in Lemmas \ref{x's as commutators} and \ref{y's
    as commutators} show that ${\mathcal U}_q(\mathfrak{n}_J)$ is generated by
    the degree one component ${\mathcal U}_q(\mathfrak{n}_J)_1$.  Hence it
    suffices to show that each of the degree one generators of ${\mathcal
    U}_q(\mathfrak{n}_J)$ gets sent to a scalar multiple of itself under the
    map $\phi$. We handle this on a case by case basis for each subset $J$ of
    simple roots. The strategy is the same in each case. One key step is to
    observe that the sets $C_d^m$ and $\gamma_{d, \ell}^m$ (defined in
    (\ref{definition C}) and (\ref{definition gamma})) are $\phi$-invariant. In
    several instances we will be able to characterize the $C_d^m$'s and
    $\gamma_{d, \ell}^m$'s, or intersections of them, as either the set of all
    scalar multiples or nonzero scalar multiples of a generator of
    $\mathcal{U}_q(\mathfrak{n}_J)$. In these cases, we can immediately
    conclude that $\phi$ sends that particular generator to a multiple of
    itself. In other cases, we can show that the $C_d^m$'s (or intersections of
    $C_d^m$'s) are vector subspaces of ${\mathcal U}_q(\mathfrak{n}_J)_1$
    spanned by either two or three generators of ${\mathcal
    U}_q(\mathfrak{n}_J)$. In these cases we will need to appeal to the
    defining relations of Lemmas \ref{x's as commutators} and \ref{y's as
    commutators} as well as Corollary \ref{LS corollary} in order to conclude
    that $\phi$ indeed sends every degree one generator of ${\mathcal
    U}_q(\mathfrak{n}_J)$ to itself. Recall that Corollary \ref{LS corollary}
    gives us sufficient conditions to conclude that certain $q$-commutators,
    say $[x_i, x_j]$ or $[y_i, y_j]$, equal $0$. Throughout this proof, we are
    tacitly applying this result whenever we state that a $q$-commutator equals
    $0$.

    \noindent \underline{$J = \left\{\alpha_1 \right\}$} In this case,
    ${\mathcal U}_q\left(\mathfrak{n}_J\right) \cong \mathbf{U}^\prime_{[2,
    16]}$.  The degree one generators are $y_2$, $y_3$, $y_4$, $y_5$, $y_6$,
    $y_7$, $y_8$, $y_{10}$, $y_{11}$, $y_{12}$, $y_{13}$, $y_{14}$, $y_{15}$,
    and $y_{16}$. We have $\mathbb{K}^\times y_2 = \gamma_{6,3}^6$,
    $\mathbb{K}^\times y_3 = \gamma_{6,2}^6$, $\mathbb{K}^\times y_4 =
    \gamma_{4,1}^4 \cap \gamma_{4,1}^2$, $\mathbb{K}^\times y_5 =
    \gamma_{4,1}^2 \cap \gamma_{4,2}^0$, $\mathbb{K}^\times y_6 =
    \gamma_{6,3}^2$, $\mathbb{K}^\times y_7 = \gamma_{6,2}^2 \cap
    \gamma_{6,1}^{-2}$, $\mathbb{K}^\times y_8 = \gamma_{6,3}^0 \cap
    \gamma_{6,1}^2$, $\mathbb{K}^\times y_{10} = \gamma_{6,3}^0 \cap
    \gamma_{6,1}^{-2}$, $\mathbb{K}^\times y_{11} = \gamma_{6,2}^{-2} \cap
    \gamma_{6,1}^2$, $\mathbb{K}^\times y_{12} = \gamma_{6,3}^{-2}$,
    $\mathbb{K}^\times y_{13} = \gamma_{4,1}^{-2} \cap \gamma_{4,2}^0$,
    $\mathbb{K}^\times y_{14} = \gamma_{4,1}^{-2} \cap \gamma_{4,1}^{-4}$,
    $\mathbb{K}^\times y_{15} = \gamma_{6,2}^{-6}$, $\mathbb{K}^\times y_{16} =
    \gamma_{6,3}^{-6}$.  Hence $\phi$ is a diagonal automorphism.

    \vspace{5mm}

    \noindent \underline{$J = \left\{\alpha_2 \right\}$} In this case,
    ${\mathcal U}_q\left(\mathfrak{n}_J\right) \cong \mathbf{U}_{[2, 21]}$.
    The degree one generators are $x_2$, $x_3$, $x_4$, $x_6$, $x_7$, $x_8$,
    $x_{13}$, $x_{15}$, $x_{17}$, $x_{19}$, $x_{20}$, and $x_{21}$. We have
    $\mathbb{K}x_2 = C_4^2 \cap C_6^2$, $\mathbb{K}x_3 = C_4^2 \cap C_6^{-2}$,
    $\mathbb{K}x_4 = C_4^1 \cap C_6^2$, $\mathbb{K}x_6 = C_4^1 \cap C_6^{-2}$,
    $\mathbb{K}x_{17} = C_4^{-1} \cap C_6^2$, $\mathbb{K}x_{19} = C_4^{-1} \cap
    C_6^{-2}$, $\mathbb{K}x_{20} = C_4^{-2} \cap C_6^2$, $\mathbb{K}x_{21} =
    C_4^{-2} \cap C_6^{-2}$, $\mathbb{K}x_7 \oplus \mathbb{K}x_{13} = C_4^0
    \cap C_6^2$, and $\mathbb{K}x_8 \oplus \mathbb{K}x_{15} = C_4^0 \cap
    C_6^{-2}$.

    Thus, $\phi(x_7) = ax_7 + bx_{13}$ for some $a, b\in \mathbb{K}$. By
    applying $\phi$ to the relation $[x_7, x_{17}] = 0$ and using $[x_{13},
    x_{17}] = 0$, we conclude $b = 0$. Hence, $\phi(x_7) \in \mathbb{K}x_7$.
    Similarly, by applying $\phi$ to the relation $[x_{13}, x_{17}] = 0$ we
    conclude $\phi(x_{13}) \in \mathbb{K}x_{13}$.  Analogously, we can conclude
    that $\phi(x_8) \in \mathbb{K}x_8$ and $\phi(x_{15}) \in \mathbb{K}x_{15}$
    by applying $\phi$ to the relations $[x_8, x_{19}] = 0$ and $[x_{15},
    x_{19}] = 0$.  Hence $\phi$ is a diagonal automorphism.

    \vspace{5mm}

    \noindent \underline{$J = \left\{\alpha_3 \right\}$} Here,
    ${\mathcal U}_q\left(\mathfrak{n}_J\right) \cong \mathbf{U}^\prime_{[4,
    23]}$.  The degree one generators are $y_4$, $y_5$, $y_{17}$, $y_{19}$,
    $y_{21}$, and $y_{23}$. We have $\mathbb{K}y_4 = C_6^1 \cap C_8^2$,
    $\mathbb{K}y_5 = C_6^{-1} \cap C_8^2$, $\mathbb{K}y_{17} = C_6^1 \cap
    C_8^0$, $\mathbb{K}y_{19} = C_6^1 \cap C_8^{-2}$, $\mathbb{K}y_{21} =
    C_6^{-1} \cap C_8^0$, and $\mathbb{K}y_{23} = C_6^{-1} \cap C_8^{-2}$.
    Therefore, $\phi$ is a diagonal automorphism.

    \vspace{5mm}

    \noindent \underline{$J = \left\{\alpha_4 \right\}$} In this case,
    ${\mathcal U}_q\left(\mathfrak{n}_J\right) \cong \mathbf{U}_{[9, 23]}$.
    The degree one generators are $x_9$, $x_{10}$, $x_{13}$, $x_{15}$,
    $x_{17}$, $x_{19}$, $x_{22}$, and $x_{23}$. We have $\mathbb{K}x_9 =
    C_6^3$, $\mathbb{K}^\times x_{10} = C_6^1 \cap \gamma_{8,2}^4$,
    $\mathbb{K}^\times x_{13} = C_6^1 \cap \gamma_{8,1}^4$, $\mathbb{K}^\times
    x_{15} = C_6^1 \cap \gamma_{8,0}^4$, $\mathbb{K}^\times x_{17} = C_6^{-1}
    \cap \gamma_{8,0}^{-4}$, $\mathbb{K}^\times x_{19} = C_6^{-1} \cap
    \gamma_{8,1}^{-4}$, $\mathbb{K}^\times x_{22} = C_6^{-1} \cap
    \gamma_{8,2}^{-4}$, and $\mathbb{K}x_{23} = C_6^{-3}$. Therefore, $\phi$ is
    a diagonal automorphism.

    \vspace{5mm}

    \noindent \underline{$J = \left\{\alpha_1, \alpha_2 \right\}$} Here,
    ${\mathcal U}_q\left(\mathfrak{n}_J\right) \cong \mathbf{U}_{[1, 21]}$.
    The degree one generators are $x_1$, $x_3$, $x_6$, $x_8$, $x_{15}$,
    $x_{19}$, and $x_{21}$. We have $\mathbb{K}x_3 = C_6^2$, $\mathbb{K}x_6 =
    C_6^1$, $\mathbb{K}x_{19} = C_6^{-1}$, $\mathbb{K}x_{21} = C_6^{-2}$, and
    $\mathbb{K}x_1 \oplus \mathbb{K}x_8 \oplus \mathbb{K}x_{15} = C_6^0$.
    Hence $\phi(x_3) \in \mathbb{K}x_3$, $\phi(x_6) \in \mathbb{K}x_6$,
    $\phi(x_{19}) \in \mathbb{K}x_{19}$, $\phi(x_{21}) \in \mathbb{K}x_{21}$,
    and there exist scalars $a_{ij} \in \mathbb{K}$, ($1 \leq i, j \leq 3$)
    such that
    \[
        \begin{split}
            \phi(x_1) &= a_{11}x_1 + a_{12}x_8 + a_{13}x_{15}\\
            \phi(x_8) &= a_{21}x_1 + a_{22}x_8 + a_{23}x_{15}\\
            \phi(x_{15}) &= a_{31}x_1 + a_{32}x_8 + a_{33}x_{15}
        \end{split}
    \]
    By applying $\phi$ to the relation $[x_6, x_{15}] = 0$ and using the
    relations $[x_6, x_8] = 0$ and $[x_1, x_6] = x_4$ to straighten unordered
    monomials, we obtain
    \[
        a_{31}\left( \left(q^2 - q\right) x_1 x_6 -q^2 x_4 \right) + a_{32}
        \left( 1 - q^{-1} \right) x_6 x_8 = 0.
    \]
    Hence $a_{31} = a_{32} = 0$. Therefore $\phi(x_{15}) \in \mathbb{K}x_{15}$.

    Next we apply $\phi$ to the relation $[x_6, x_8] = 0$ and use the relations
    $[x_1, x_6] = x_4$ and $[x_6, x_{15}] = 0$ to straighten unordered
    monomials to conclude that $-q^2a_{21} x_4 + a_{23} \left( 1 - q \right)
    x_6 x_{15} = 0$.  Hence $a_{21} = a_{23} = 0$. Thus, $\phi(x_8) \in
    \mathbb{K}x_8$.

    The relations $[x_1, x_{19}] = x_{17}$ and $[x_8 , x_{17}]
    = 0$ imply $[x_8, [x_1, x_{19}]] = 0$. By applying $\phi$ to this relation
    and using the relations $[x_8 , x_{15}] = [x_8, x_{19}] = [x_{15}, x_{19}]
    = 0$ to straighten any unordered monomials, we get
    \[
        a_{12}\left(1 - q^{-4} \right) x_8^2 x_{19} + a_{13} \left( 1 - q^{-3}
        \right) x_8 x_{15} x_{19} = 0.
    \]
    Hence $a_{12} = a_{13} = 0$. Thus, $\phi(x_1) \in \mathbb{K} x_1$ and we
    conclude that $\phi$ is a diagonal automorphism.

    \vspace{5mm}

    \noindent \underline{$J = \left\{\alpha_1, \alpha_3 \right\}$} In this
    situation, ${\mathcal U}_q\left(\mathfrak{n}_J\right) \cong
    \mathbf{U}^\prime_{[2, 23]}$.  The degree one generators are $y_2$, $y_3$,
    $y_{17}$, $y_{19}$, $y_{21}$, and $y_{23}$. We have $\mathbb{K}y_2 = C_8^0
    \cap C_{22}^2$, $\mathbb{K}y_3 = C_8^0 \cap C_{22}^{-2}$, $\mathbb{K}y_{17}
    = C_8^1 \cap C_{22}^2$, $\mathbb{K}y_{19} = C_8^1 \cap C_{22}^{-2}$,
    $\mathbb{K}y_{21} = C_8^{-1} \cap C_{22}^2$, and $\mathbb{K}y_{23} =
    C_8^{-1} \cap C_{22}^{-2}$. Hence $\phi$ is a diagonal automorphism.

    \vspace{5mm}

    \noindent \underline{$J = \left\{\alpha_1, \alpha_4 \right\}$} In this
    case, ${\mathcal U}_q\left(\mathfrak{n}_J\right) \cong
    \mathbf{U}^\prime_{[1, 20]}$.  The degree one generators are $y_1$, $y_2$,
    $y_3$, $y_5$, $y_{11}$, $y_{12}$, $y_{17}$, $y_{19}$, and $y_{20}$. We have
    $\mathbb{K}^\times y_1 = C_{10}^2 \cap \gamma_{14, 1}^2$,
    $\mathbb{K}^\times y_2 = C_{10}^2 \cap \gamma_{14, 1}^4$, $\mathbb{K} y_3 =
    C_{10}^2 \cap C_{14}^0$, $\mathbb{K} y_5 = C_{10}^0 \cap C_{14}^0$,
    $\mathbb{K} y_{11} = C_{10}^{-2} \cap C_{14}^0$, $\mathbb{K}^\times y_{12}
    = C_{10}^{-2} \cap \gamma_{14, 1}^{-4}$, $\mathbb{K}^\times y_{17} =
    C_{10}^0 \cap \gamma_{14, 1}^2$, $\mathbb{K}^\times y_{19} = C_{10}^0 \cap
    \gamma_{14, 1}^{-2}$, and $\mathbb{K}^\times y_{20} = C_{10}^{-2} \cap
    \gamma_{14, 1}^{-2}$. Hence $\phi$ is a diagonal automorphism.

    \vspace{5mm}

    \noindent \underline{$J = \left\{\alpha_2, \alpha_3 \right\}$} Here,
    ${\mathcal U}_q\left(\mathfrak{n}_J\right) \cong \mathbf{U}^\prime_{[3,
    24]}$.  The degree one generators are $y_3$, $y_{17}$, $y_{21}$, and
    $y_{24}$. We have $\mathbb{K}y_3 = C_{14}^2$, $\mathbb{K}y_{17} =
    C_{10}^1$, $\mathbb{K}y_{21} = C_{10}^{-1}$, and $\mathbb{K}y_{24} =
    C_{14}^{-2}$.  Hence, $\phi$ is a diagonal automorphism.

    \vspace{5mm}

    \noindent \underline{$J = \left\{\alpha_2, \alpha_4 \right\}$} In this
    case, ${\mathcal U}_q\left(\mathfrak{n}_J\right) \cong \mathbf{U}_{[2,
    23]}$.  The degree one generators are $x_2$, $x_3$, $x_4$, $x_6$, $x_7$,
    $x_8$, $x_{22}$, and $x_{23}$. We have $\mathbb{K}x_2 = C_{10}^2 \cap
    C_{14}^2$, $\mathbb{K}x_3 = C_{10}^{-2} \cap C_{14}^2$, $\mathbb{K}x_4 =
    C_{10}^2 \cap C_{14}^0$, $\mathbb{K}x_6 = C_{10}^{-2} \cap C_{14}^0$,
    $\mathbb{K}x_7 = C_{10}^2 \cap C_{14}^{-2}$, $\mathbb{K}x_8 = C_{10}^{-2}
    \cap C_{14}^{-2}$, $\mathbb{K}x_{22} = C_{14}^1$, and $\mathbb{K}x_{23} =
    C_{14}^{-1}$.  Hence, $\phi$ is a diagonal automorphism.

    \vspace{5mm}

    \noindent \underline{$J = \left\{\alpha_3, \alpha_4 \right\}$} We have
    ${\mathcal U}_q\left(\mathfrak{n}_J\right) \cong \mathbf{U}_{[4, 24]}$.
    The degree one generators are $x_4$, $x_6$, $x_{22}$, and $x_{24}$.  We
    have $\mathbb{K}x_4 = C_{12}^2$, $\mathbb{K}x_{24} = C_{12}^{-2}$, and
    $\mathbb{K}x_{6} \oplus \mathbb{K}x_{22} = C_{12}^0$. Thus, there exist
    scalars $a_{11}, a_{12}, a_{21}, a_{22}, \gamma, \delta \in \mathbb{K}$
    such that $\phi (x_6) = a_{11}x_6 + a_{12} x_{22}$, $\phi(x_{22}) = a_{21}
    x_6 + a_{22} x_{22}$, $\phi(x_4) = \gamma x_4$, and $\phi(x_{24}) = \delta
    x_{24}$.

    Observe first that the relation $[x_4, x_{24}] = (q + q^{-1})x_7$ implies
    $\phi(x_7) = \gamma \delta x_7$. Next apply $\phi$ to the relation $[x_6,
    x_7] = 0$ and use the relation $[x_7, x_{22}] = x_{17}$ to straighten
    unordered monomials to obtain
    \[
        a_{12} \left( \left(q^2 - 1 \right) x_7 x_{22} - q^2 x_{17} \right) =
        0.
    \]
    Hence $a_{12} = 0$. Thus $\phi(x_6) \in \mathbb{K}x_6$.

    The relations $[x_4, x_{22}] = x_{13}$ and $[x_4, x_{13}] = 0$ imply $[x_4,
    [x_4, x_{22}]] = 0$. Applying $\phi$ to this relation and using the
    relations $[x_4, x_6] = \left(q + q^{-1}\right) x_5$ and $[x_4, x_5] = 0$
    to straighten unordered monomials we get
    \[
        a_{21} \left( 1 - q^{-1} \right) \left( \left(1 - q\right) x_4^2x_6 +
        \left(q + q^{-1}\right)^2 x_4 x_5 \right) = 0.
    \]
    Hence $a_{21} = 0$. Therefore $\phi(x_{22}) \in \mathbb{K} x_{22}$. Thus
    $\phi$ is a diagonal automorphism.

    \vspace{5mm}

    \noindent \underline{$J = \left\{\alpha_1, \alpha_2, \alpha_3 \right\}$} In
    this case, ${\mathcal U}_q\left(\mathfrak{n}_J\right) \cong
    \mathbf{U}^\prime_{[2, 24]}$.  The degree one generators are $y_2$,
    $y_{17}$, $y_{21}$, and $y_{24}$. We have $\mathbb{K}y_{17} = C_{12}^1$,
    $\mathbb{K}y_{21} = C_{12}^{-1}$, and $\mathbb{K}y_2 \oplus
    \mathbb{K}y_{24} = C_{12}^0$. Thus, there exist $b_2, b_{24}, c_2, c_{24}
    \in \mathbb{K}$ such that $\phi(y_2) = b_2y_2 + b_{24}y_{24}$ and
    $\phi(y_{24}) = c_2 y_2 + c_{24}y_{24}$. Applying $\phi$ to the relation
    $[y_2, y_{17}] = 0$ and using the identity $[y_{17}, y_{24}] = y_{19}$
    gives us
    \[
        0 = (b_2y_2 + b_{24}y_{24})y_{17} - y_{17}(b_2y_2 + b_{24}y_{24}) =
        b_{24}\left(q\widehat{q}y_{17}y_{24} - q^2y_{19}\right).
    \]
    Hence $b_{24} = 0$. Next we observe that $[y_2, y_{24}] = y_3$ and $[y_2,
    y_3] = 0$.  Hence we have the relation $[y_2, [y_2, y_{24}]] = 0$.
    Applying $\phi$ to this relation gives us $c_2\widehat{q}^2y_2^3 = 0$.
    Therefore $c_2 = 0$ and we conclude that $\phi$ is a diagonal automorphism.

    \vspace{5mm}

    \noindent \underline{$J = \left\{\alpha_1, \alpha_2, \alpha_4 \right\}$} In
    this case, ${\mathcal U}_q\left(\mathfrak{n}_J\right) \cong \mathbf{U}_{[1,
    23]}$.  The degree one generators are $x_1$, $x_3$, $x_6$, $x_8$, $x_{22}$,
    and $x_{23}$. We have $\mathbb{K}x_3, = C_{18}^2$, $\mathbb{K}x_8 =
    C_{18}^{-2}$, $\mathbb{K}x_{22} = C_{18}^1$, $\mathbb{K}x_{23} =
    C_{18}^{-1}$, and $\mathbb{K}x_1 \oplus \mathbb{K}x_6 = C_{18}^0$. Hence,
    there exist $b_1, b_6, c_1, c_6 \in \mathbb{K}$ such that $\phi(x_1) =
    b_1x_1 + b_6x_6$ and $\phi(x_6) = c_1x_1 + c_6x_6$. Applying $\phi$ to the
    relation $[x_1, x_{22}] = 0$ gives us
    \[
        0 = (b_1x_1 + b_6x_6)x_{22} - x_{22}(b_1x_1 + b_6x_6) = b_6\left( (1 -
        q) x_6x_{22} + qx_{15}\right).
    \]
    because $[x_6, x_{22}] = x_{15}$.  Hence $b_6 = 0$.  Similarly, by applying
    $\phi$ to the relation $[x_3, x_6] = 0$ we obtain
    \[
        0 = x_3(c_1x_1 + c_6x_6) - q^2(c_1x_1 + c_6x_6)x_3 = -c_1q^2x_2
    \]
    because $[x_1, x_3] = x_2$. Hence $c_1 = 0$. Therefore $\phi$ is a diagonal
    automorphism.

    \vspace{5mm}

    \noindent \underline{$J = \left\{\alpha_1, \alpha_3, \alpha_4 \right\}$} In
    this situation, ${\mathcal U}_q\left(\mathfrak{n}_J\right) \cong
    \mathbf{U}^\prime_{[1, 23]}$.  The degree one generators are $y_1$, $y_2$,
    $y_3$, $y_{21}$, and $y_{23}$.  We have $\mathbb{K}y_1 = C_{30}^0$,
    $\mathbb{K}y_2 \oplus \mathbb{K}y_{21} = C_{30}^2$, and $\mathbb{K}y_3
    \oplus \mathbb{K}y_{23} = C_{30}^{-2}$. Hence there exist $a_2, a_{21},
    b_2, b_{21}, c_3, c_{23}, d_3, d_{23} \in \mathbb{K}$ such that $\phi(y_2)
    = a_2y_2 + a_{21}y_{21}$, $\phi(y_{21}) = b_2y_2 + b_{21}y_{21}$,
    $\phi(y_3) = c_3y_3 + c_{23}y_{23}$, and $\phi(y_{23}) = d_3y_3 +
    d_{23}y_{23}$. The relations $[y_1, y_{17}] = 0$ and $[y_1, y_{21}] =
    y_{17}$ give us $[y_1, [y_1, y_{21}]] = 0$. Applying $\phi$ to this
    relation and using the commutation relation $y_1y_2 = y_2y_1$ gives us
    $b_2(2 - q - q^{-1})y_1^2y_2 = 0$. Hence $b_2 = 0$. Next, since $[y_1,
    y_{21}] = y_{17}$ and $[y_2, y_{17}] = 0$, we have the relation $[y_2,
    [y_1, y_{21}]] = 0$.  Applying $\phi$ to this relation and using the
    identity $[y_{17}, y_{21}] = 0$ gives us $a_{21}\left(q^{-1} - 1\right)
    y_{17}y_{21} = 0$. Therefore $a_{21} = 0$. Observe next that $[y_{21},
    y_{23}] = [2]_qy_{22}$ and $[y_{21}, y_{22}] = 0$.  Therefore $[y_{21},
    [y_{21}, y_{23}]] = 0$. Applying $\phi$ to this relation and using the
    identities $[y_3, y_{21}] = y_5$ and $[y_5, y_{21}] = [2]_qy_{11}$ give us
    $d_3[2]_qq^2y_{11} = 0$.  Thus $d_3 = 0$.  Finally, applying $\phi$ to the
    relation $[y_1, y_3] = 0$ and using the identity $[y_1, y_{23}] = y_{19}$
    gives us $c_{23}\left( (1 - q)y_1y_{23} + qy_{19} \right) = 0$. Hence
    $c_{23} = 0$.  Therefore $\phi$ is a diagonal automorphism.

    \vspace{5mm}

    \noindent \underline{$J = \left\{ \alpha_2, \alpha_3, \alpha_4 \right\}$}
    In this case, ${\mathcal U}_q\left(\mathfrak{n}_J\right) \cong
    \mathbf{U}_{[2, 24]}$.  The degree one generators are $x_2$, $x_3$,
    $x_{22}$, and $x_{24}$. We have $\mathbb{K}x_2 = C_{18}^2$, $\mathbb{K}x_3
    = C_{18}^{-2}$, and $\mathbb{K}x_{22} \oplus \mathbb{K}x_{24} = C_{18}^0$.
    By applying $\phi$ to the relation $[x_2, x_{22}] = 0$ and using the
    relation $[x_2, x_{24}] = x_4$ to straighten unordered monomials, we can
    conclude that $\phi(x_{22}) \in \mathbb{K}x_{22}$.

    Since $\mathbb{K}x_{22} \oplus \mathbb{K}x_{24}$ is a $\phi$-invariant
    subspace, there exist scalars $\gamma, \delta \in \mathbb{K}$ such that
    $\phi(x_{24}) = \gamma x_{22} + \delta x_{24}$. The relations $[x_2,
    x_{24}] = x_4$ and $[x_3, x_4] = 0$ give us the relation $[x_3, [x_2,
    x_{24}]] = 0$. Applying $\phi$ to this relation and using the relations
    $[x_2, x_{22}] = [x_2, x_3] = [x_3, x_{22}] = 0$ to straighten any
    unordered monomials gives us $\gamma \left(q^{-2} - 1\right) x_2 x_3 x_{22}
    = 0$. Hence $\gamma = 0$. Therefore $\phi(x_{24}) \in \mathbb{K}x_{24}$ and
    $\phi$ is a diagonal automorphism.

\end{proof}

We are now able to prove the main result of this section. The following theorem
proves Conjecture \ref{conj} when the underlying Lie algebra
$\mathfrak{g}$ is of type $F_4$.

\begin{theorem}

    \label{Theorem, F4}

    If $\mathfrak{g}$ is the Lie algebra of type $F_4$ and $J$ is a nonempty
    subset of simple roots of $\mathfrak{g}$, then
    $\operatorname{Aut}({\mathcal U}_q(\mathfrak{n}_J)) \cong
    \left(\mathbb{K}^\times\right)^4$.

\end{theorem}

\begin{proof}

    By Theorem \ref{F4: diagonal}, every automorphism of ${\mathcal
    U}_q(\mathfrak{n}_J)$ is a diagonal automorphism. As a CGL extension, the
    algebra ${\mathcal U}_q(\mathfrak{n}_J)$ has rank $4$. Thus, by
    \cite[Theorems 5.3 and 5.5]{GY2}, $\operatorname{Aut}({\mathcal
    U}_q(\mathfrak{n}_J)) \cong \left(\mathbb{K}^\times\right)^4$.

\end{proof}

\section{Two lemmas regarding ${\mathcal U}_q(\mathfrak{n}_J)$ when
$\mathfrak{g} = F_4$}

In this section we prove two lemmas regarding the quantized nilradicals
${\mathcal U}_q(\mathfrak{n}_J)$ for the case when the underlying Lie algebra
$\mathfrak{g}$ is of type $F_4$ and $J$ is any nonempty subset of simple roots.

Recall from Section \ref{Z grading section} that for each coweight $\lambda \in
P^\vee$, there is an induced $\mathbb{Z}$-grading on ${\mathcal U}_q\left(
\mathfrak{n}_J \right)$. We will use the coweight $\lambda = \sum_{i\in J}
\varpi_i^\vee \in P^\vee$.  The two lemmas in this section explicitly show how
every Lusztig root vector $X_\beta$ in ${\mathcal U}_q(\mathfrak{n}_J)$ with
$\operatorname{height}(\beta) > 1$ can be written, up to a scalar multiple, as
a $q$-commutator of other Lusztig root vectors. As a direct consequence of
these lemmas, one can readily verify that each quantized nilradical ${\mathcal
U}_q\left(\mathfrak{n}_J\right)$ is generated, as an algebra, by the Lusztig
root vectors of degree $1$. Hence, ${\mathcal U}_q(\mathfrak{n}_J)$ is a
locally finite, connected, $\mathbb{N}$-graded algebra generated by its graded
component of degree $1$.

\begin{lemma}
    \label{x's as commutators}

    Let $\mathfrak{g}$ be the Lie algebra of type $F_4$, and let $x_1,\dots,
    x_{24}$ be the Lusztig root vectors  (recall (\ref{xi, yi definition}))
    corresponding to the reduced expression $\mathbf{R}[w_o]$ (see (\ref{two
    reduced expressions})) of the longest element of the Weyl group of
    $\mathfrak{g}$. Then

    \begin{enumerate}

        \begin{multicols}{3}

            \item[] $x_2 = [x_1, x_3]$,
            \item[] $x_4 = [x_1, x_6]$,
            \item[] $x_4 = [x_2, x_{24}]$,
            \item[] $x_5 = \frac{1}{[2]_q} [x_4, x_6]$,
            \item[] $x_6 = [x_3, x_{24}]$,
            \item[] $x_7 = [x_1, x_8]$,
            \item[] $x_7 = \frac{1}{[2]_q}[x_4, x_{24}]$,
            \item[] $x_8 = \frac{1}{[2]_q}[x_6, x_{24}]$,
            \item[] $x_9 = [x_6, x_{13}]$,
            \item[] $x_{10} = [x_8, x_{13}]$,
            \item[] $x_{11} = \frac{1}{[2]_q}[x_{10}, x_{13}]$,
            \item[] $x_{12} = \frac{1}{[2]_q}[x_{10}, x_{15}]$,
            \item[] $x_{13} = [x_1, x_{15}]$,
            \item[] $x_{13} = [x_4, x_{22}]$,
            \item[] $x_{14} = \frac{1}{[2]_q}[x_{13}, x_{15}]$,
            \item[] $x_{15} = [x_6, x_{22}]$,
            \item[] $x_{16} = [x_{15}, x_{17}]$,
            \item[] $x_{17} = [x_1, x_{19}]$,
            \item[] $x_{17} = [x_7, x_{22}]$,
            \item[] $x_{18} = \frac{1}{[2]_q}[x_{17}, x_{19}]$,
            \item[] $x_{19} = [x_8, x_{22}]$,
            \item[] $x_{20} = [x_1, x_{21}]$,
            \item[] $x_{20} = \frac{1}{[2]_q}[x_{17}, x_{22}]$,
            \item[] $x_{21} = \frac{1}{[2]_q}[x_{19}, x_{22}]$,
            \item[] $x_{23} = [x_{22}, x_{24}]$.

        \end{multicols}

    \end{enumerate}

\end{lemma}

\begin{proof}

    Throughout the proof of this lemma, we adopt the abbreviation
    \[
        E_{ij} := [E_i, E_j],  \hspace{10mm} i,j \in \mathbf{I},
    \]
    for a $q$-commutator and inductively define the nested $q$-commutator
    \[
        E_{i_1 i_2 \cdots i_n} := [E_{i_1 i_2 \cdots i_{n - 1}}, E_{i_n}]
    \]
    for $i_1, i_2, \dots, i_n \in \mathbf{I}$.

    Since $s_1$ is the first simple reflection appearing the reduced expression
    $R[w_o]$, it is clear that $x_1 = E_1$. As a direct consequence of
    Proposition \ref{proposition 1}, $x_3 = E_2$, $x_{22} = E_4$, and $x_{24} =
    E_3$. At this point we have identified each Chevalley generator $E_i$ with
    a corresponding Lusztig root vector. Next we will show how the remaining
    Lusztig root vectors $x_i$ can be written, up to a scalar multiple, as
    $q$-commutators of other Lusztig root vectors. The next step will be to
    focus on those $x_i$ with $\langle \operatorname{deg}_Q(x_i), \varpi_1^\vee
    + \varpi_2^\vee + \varpi_3^\vee + \varpi_4^\vee \rangle = 2$. In other
    words, we focus on the $x_i$'s such that $\operatorname{deg}_Q(x_i)$ has
    height $2$. For short, we will say that the $\mathbb{N}$-degree of $x_i$ is
    $d$ whenever $\operatorname{deg}_Q(x_i)$ has height $d$. We compute
    \begin{align*}
        x_2 &= \mathbf{T}_1(E_2) = E_{12} = [x_1, x_3],
        \\
        x_6 &= \mathbf{T}_{\underline{121}32}(E_3) = \mathbf{T}_{2 \cdot
        1232}(E_3) = \mathbf{T}_2(E_3) = E_{23} = [x_3, x_{24}],
        \\
        x_{23} &=
        \mathbf{T}_{\underline{1213231243213234323123}}(E_4)
        = \mathbf{T}_{4 \cdot 121321324321323432132}(E_4) =
        \mathbf{T}_4(E_3)
        \\
        &= E_{43} = [x_{22}, x_{24}],
    \end{align*}
    where, in the above computations, we adopt the underlining notation, as in
    $\mathbf{T}_{\underline{121}32}$ above, to highlight that braid relations
    in the Weyl group are being applied to the underlined part in moving from
    one step in the calculations to the next. We use the dot notation, as in
    $\mathbf{T}_{2 \cdot 1232}$ above, to split a reduced word into two parts
    in order to indicate which Lusztig symmetries are being applied at that
    particular step. We continue to compute
    \begin{align*}
        x_4 &= \mathbf{T}_{12\cdot 1}(E_3) = \mathbf{T}_{1\cdot 2}(E_3) =
        T_1(E_{23}) = E_{123} = [E_{12}, E_3] = [x_2, x_{24}],
        \\
        [2]_q x_8 &= [2]_q \mathbf{T}_{\underline{121323}1}(E_2)
        = [2]_q \mathbf{T}_{23\cdot 12321}(E_2)
        = [2]_q \mathbf{T}_{23}(E_2)
        = E_{233} = [E_{23}, E_3]
        \\
        &= [x_6, x_{24}],
        \\
        x_{15} &= \mathbf{T}_{\underline{121323124321}32}(E_3)
        = \mathbf{T}_{23 \cdot 123412321232}(E_3)
        = \mathbf{T}_{23}(E_4)
        = E_{234}
        \\
        &= [E_{23}, E_4] = [x_6, x_{22}].
    \end{align*}
    So far we have identified how each Lusztig root vector $x_i$ with
    $\mathbb{N}$-degree at most $3$ can be written as a $q$-commutator of
    Lusztig root vectors of smaller $\mathbb{N}$-degree. We continue in this
    manner focusing next on the Lusztig root vectors $x_i$ having
    $\mathbb{N}$-degree equal to $4$. We compute
    \begin{align*}
        [2]_q x_7 &= [2]_q \mathbf{T}_{12\underline{13}23}(E_1) =
        [2]_q \mathbf{T}_{123\cdot 123}(E_1)
        = [2]_q \mathbf{T}_{1\cdot 23}(E_2)
        = \mathbf{T}_1(E_{233})
        \\
        &= E_{1233} = [E_{123}, E_3] = [x_4, x_{24}],
        \\
        x_{13} &= \mathbf{T}_{12\underline{13}23124321}(E_3) =
        \mathbf{T}_{123 \cdot 123124321}(E_3)
        = \mathbf{T}_{123}(E_4)
        = E_{1234} = [E_{123}, E_4]
        \\
        &= [x_4, x_{22}],
        \\
        [2]_q x_{19} &= [2]_q
        \mathbf{T}_{\underline{121323124321323}432}(E_3) = [2]_q
        \mathbf{T}_{2342 \cdot 12321432132432}(E_3)
        \\
        &= [2]_q \mathbf{T}_{234\cdot 2}(E_3) = [2]_q
        \mathbf{T}_{234}(E_{23}) = [2]_q [\mathbf{T}_{23\cdot 4}(E_2),
        \mathbf{T}_{234}(E_3)]
        \\
        &= [2]_q [\mathbf{T}_{23}(E_2), E_4] = [E_{233}, x_{22}] = [2]_q [x_8,
        x_{22}],
        \\
        x_{19} &= \mathbf{T}_{\underline{121323124321323}432}(E_3)
        = \mathbf{T}_{2324 \cdot 12321432132432}(E_3)
        \\
        &= \mathbf{T}_{232\cdot 4}(E_3)
        = \mathbf{T}_{232}(E_{43})
        = [\mathbf{T}_{232}(E_4), \mathbf{T}_{232}(E_3)]
        = [E_{234}, E_3] = [x_{15}, x_{24}].
    \end{align*}
    Next we show how each Lusztig root vector $x_i$ of $\mathbb{N}$-degree
    equal to $5$ (i.e. $x_5$, $x_{17}$, and $x_{21}$) can be written as a
    $q$-commutator of Lusztig root vectors of smaller $\mathbb{N}$-degree.  We
    have
    \begin{align*}
        [2]_q x_5 &=
        [2]_q\mathbf{T}_{121\cdot 3}(E_2)
        =\mathbf{T}_{121}([E_3, E_{32}])
        =\mathbf{T}_{121}([E_3, \mathbf{T}_{32}(E_3)])
        \\
        &= [\mathbf{T}_{121}(E_3), \mathbf{T}_{12132}(E_3)]
        = [x_4, x_6],
        \\
        x_{17} &= \mathbf{T}_{12\underline{13231243213}234}(E_3)
        = \mathbf{T}_{12324\cdot 12324321234}(E_3)
        = \mathbf{T}_{1232\cdot 4}(E_3)
        \\
        &=\mathbf{T}_{1232}(E_{43})
        = [\mathbf{T}_{123 \cdot 2}(E_4), \mathbf{T}_{1232}(E_3)]
        = [\mathbf{T}_{123}(E_4), E_3]
        = [E_{1234}, E_3]
        \\
        &= [x_{13}, x_{24}],
        \\
        [2]_q x_{17} &= [2]_q\mathbf{T}_{12\underline{13231243213}234}(E_3)
        = [2]_q\mathbf{T}_{12324\cdot 12324321234}(E_3)
        = [2]_q\mathbf{T}_{123\underline{24}}(E_3)
        \\
        &=[2]_q\mathbf{T}_{1234\cdot 2}(E_3)
        =[2]_q\mathbf{T}_{1234}([E_2, E_3])
        =[2]_q[\mathbf{T}_{123\cdot 4}(E_2), \mathbf{T}_{1234}(E_3)]
        \\
        &=[2]_q[\mathbf{T}_{1\cdot 23}(E_2), E_4]
        =[\mathbf{T}_{1}(E_{233}), E_4]
        =[E_{1233}, E_4]
        =[2]_q [x_7, x_{22}],
        \\
        [2]_q x_{21} &= [2]_q
        \mathbf{T}_{\underline{1213231243213234323}1}(E_2)
        =
        [2]_q \mathbf{T}_{23243\cdot 123241321324321}(E_2)
        \\
        &= [2]_q\mathbf{T}_{2324\cdot 3}(E_2)
        = \mathbf{T}_{2324}([E_3, E_{32}])
        = [\mathbf{T}_{232\cdot 4}(E_3),
        \mathbf{T}_{23\underline{24}}(E_{32})]
        \\
        &=
        [\mathbf{T}_{232}(E_{43}), \mathbf{T}_{234\cdot 2}(E_{32})]
        =  [[\mathbf{T}_{23 \cdot 2}(E_4),
        \mathbf{T}_{232}(E_3)], \mathbf{T}_{234}(E_3)]
        \\
        &= [[\mathbf{T}_{23}(E_4), E_3], E_4]
        = [E_{234}, E_3], E_4]
        = [E_{2343}, E_4] = [x_{19}, x_{22}].
    \end{align*}
    Continuing in this manner, we get
    \begin{align*}
        x_9 &= \mathbf{T}_{121323 \cdot 12}(E_4)
        = \mathbf{T}_{12132\cdot 3}(E_4)
        = \mathbf{T}_{12132}(E_{34})
        = [\mathbf{T}_{12132}(E_3), \mathbf{T}_{12\underline{13}2}(E_4)]
        \\
        &= [x_6, \mathbf{T}_{123\cdot 12}(E_4)]
        = [x_6, \mathbf{T}_{123}(E_4)]
        = [x_6, x_{13}],
        \\
        x_{10} &= \mathbf{T}_{1213231\underline{24}}(E_3)
        = \mathbf{T}_{12132314 \cdot 2}(E_3)
        = \mathbf{T}_{12132314}(E_{23})
        \\
        &= [\mathbf{T}_{1213231 \cdot 4}(E_2),
        \mathbf{T}_{1213\cdot 2314}(E_3)]
        = [\mathbf{T}_{1213231}(E_2), \mathbf{T}_{12\underline{13}}(E_4)]
        \\
        &= [x_8, \mathbf{T}_{123 \cdot 1}(E_4)]
        =  [x_8, \mathbf{T}_{123}(E_4)]
        =  [x_8, x_{13}],
        \\
        [2]_q x_{11} &= [2]_q \mathbf{T}_{121323124\cdot 3}(E_2) =
        \mathbf{T}_{121323124}([E_3, E_{32}])
        \\
        &= [\mathbf{T}_{121323124}(E_3),
        \mathbf{T}_{12\underline{1323124}}(E_{32})]
        = [x_{10}, \mathbf{T}_{12312314 \cdot 2}(E_{32})]
        \\
        &= [x_{10}, \mathbf{T}_{123\cdot 12314}(E_3)]
        = [x_{10}, \mathbf{T}_{123}(E_4)]
        = [x_{10}, x_{13}],
        \\
        [2]_q x_{14} &= [2]_q \mathbf{T}_{121323124321\cdot 3}(E_2)
        = \mathbf{T}_{121323124321} ([E_3, E_{32}])
        \\
        &= [\mathbf{T}_{121323124321}(E_3),
        \mathbf{T}_{121323124321}(E_{32})]
        = [x_{13},
        \mathbf{T}_{121323124321}\mathbf{T}_{32}(E_3)]
        \\
        &= [x_{13}, x_{15}],
        \\
        [2]_q x_{12} &= [2]_q \mathbf{T}_{\underline{12132312432}}(E_1)
        = [2]_q \mathbf{T}_{231213423\cdot 12}(E_1)
        = [2]_q \mathbf{T}_{2312134\cdot 23}(E_2)
        \\
        &= \mathbf{T}_{2312134}(E_{233})
        = [\mathbf{T}_{2312134}(E_{23}), \mathbf{T}_{23\cdot 12134}(E_3)]
        \\
        &= [\mathbf{T}_{23121342}(E_3), \mathbf{T}_{23}(E_4)]
        = [\mathbf{T}_{23121342\cdot 1}(E_3), E_{234}]
        \\
        &= [\mathbf{T}_{\underline{121323124}}(E_3), E_{234}]
        = [x_{10}, x_{15}],
        \\
        x_{16} &= \mathbf{T}_{12132312432132\cdot 3}(E_4) =
        \mathbf{T}_{12132312432132}(E_{34})
        \\
        &=
        [\mathbf{T}_{12132312432132}(E_3),
        \mathbf{T}_{12132312432132}(E_4)]
        \\
        &=
        [\mathbf{T}_{12132312432132}(E_3),
        \mathbf{T}_{12132312432132}\mathbf{T}_{34}(E_3)]
        = [x_{15}, x_{17}],
        \\
        [2]_q x_{18} &= [2]_q \mathbf{T}_{1213231243213234\cdot 3}(E_2) =
        \mathbf{T}_{1213231243213234}([E_3, E_{32}])
        \\
        &= [\mathbf{T}_{1213231243213234}(E_3),
        \mathbf{T}_{1213231243213234}(E_{32})]
        \\
        &= [x_{17}, \mathbf{T}_{1213231243213234}\mathbf{T}_{32}(E_3)]
        = [x_{17}, x_{19}],
        \\
        [2]_qx_{20} &=
        [2]_q\mathbf{T}_{12\underline{1323124321323432}3}(E_1)
        = [2]_q\mathbf{T}_{123243\cdot 1232143231423}(E_1)
        \\
        &= [2]_q\mathbf{T}_{12324\cdot 3}(E_2) = \mathbf{T}_{12324}([E_3,
        E_{32}]) = [\mathbf{T}_{1232\cdot 4}(E_3),
        \mathbf{T}_{123\underline{24}}(E_{32})]
        \\
        &=
        [\mathbf{T}_{1232}(E_{43}), \mathbf{T}_{1234 \cdot 2}(E_{32})]
        =  [[\mathbf{T}_{123 \cdot 2}(E_4),
        \mathbf{T}_{1232}(E_3)], \mathbf{T}_{1234}(E_3)]
        \\
        &= [[\mathbf{T}_{123}(E_4), E_3], E_4]
        = [E_{1234}, E_3], E_4]
        = [E_{12343}, E_4] = [x_{17}, x_{22}].
    \end{align*}
    Finally, we can use $q$-associativity to prove the remaining identities,
    \begin{align*}
        &x_4 = [x_2, x_{24}] = [[x_1, x_3], x_{24}] = [x_1, [x_3, x_{24}]]
        = [x_1, x_6],
        \\
        &[2]_q x_7 = [x_4, x_{24}] = [[x_1, x_6], x_{24}] =
        [x_1, [x_6, x_{24}]] = [2]_q [x_1, x_8],
        \\
        &x_{13} = [x_4, x_{22}] = [[x_1, x_6], x_{22}] = [x_1, [x_6,
        x_{22}]] = [x_1, x_{15}],
        \\
        &x_{17} = [x_7, x_{22}] = [[x_1, x_8], x_{22}] = [x_1, [x_8,
        x_{22}]] = [x_1, x_{19}],
        \\
        &[2]_q x_{20} = [x_{17}, x_{22}] = [[x_1, x_{19}], x_{22}] = [x_1,
        [x_{19}, x_{22}]] = [2]_q [x_1, x_{21}].
    \end{align*}

\end{proof}

\begin{lemma}
    \label{y's as commutators}

    Let $\mathfrak{g}$ be the Lie algebra of type $F_4$, and let $y_1,\dots,
    y_{24}$ be the Lusztig root vectors  (recall (\ref{xi, yi definition}))
    corresponding to the reduced expression $\mathbf{R}^\prime[w_o]$ (see
    (\ref{two reduced expressions})) of the longest element of the Weyl group
    of $\mathfrak{g}$. Then

    \begin{enumerate}

         \begin{multicols}{3}

            \item[] $y_3 = [y_2, y_{24}]$,

            \item[] $y_4 = [y_3, y_{17}]$,

            \item[] $y_5 = [y_3, y_{21}]$,

            \item[] $y_6 = \frac{1}{[2]_q} [y_4, y_{17}]$,

            \item[] $y_7 = \frac{1}{[2]_q} [y_4, y_{19}]$,

            \item[] $y_8 = [y_5, y_{17}]$,

            \item[] $y_9 = \frac{1}{[2]_q} [y_8, y_{10}]$,

            \item[] $y_{10} = [y_5, y_{19}]$,

            \item[] $y_{11} = \frac{1}{[2]_q} [y_5, y_{21}]$,

            \item[] $y_{12} = \frac{1}{[2]_q} [y_5, y_{23}]$,

            \item[] $y_{13} = [y_{10}, y_{17}]$,

            \item[] $y_{14} = [y_{12}, y_{17}]$,

            \item[] $y_{15} = \frac{1}{[2]_q}[y_{14}, y_{17}]$,

            \item[] $y_{16} = \frac{1}{[2]_q}[y_{14}, y_{19}]$,

            \item[] $y_{17} = [y_1, y_{21}]$,

            \item[] $y_{18} = \frac{1}{[2]_q}[y_{17}, y_{19}]$,

            \item[] $y_{19} = [y_1, y_{23}]$,

            \item[] $y_{19} = [y_{17}, y_{24}]$,

            \item[] $y_{20} = [y_{19}, y_{21}]$,

            \item[] $y_{22} = \frac{1}{[2]_q}[y_{21}, y_{23}]$,

            \item[] $y_{23} = [y_{21}, y_{24}]$.

         \end{multicols}

     \end{enumerate}

\end{lemma}

\begin{proof}

    In the proof of this lemma we adopt the same abbreviation for
    $q$-commutators as used in the proof of Lemma \ref{x's as commutators}.

    Observe first that the two reduced expressions for $w_o$ in (\ref{two
    reduced expressions}) share a common substring, namely ${\mathbf
    R}^\prime[6, 23] = {\mathbf R}[2, 19]$. Hence, there is an algebra
    isomorphism $\Phi: \mathbf{U}^\prime_{[6,23]} \to \mathbf{U}_{[2,19]}$ such
    that $\Phi(y_i) = x_{i - 4}$ for all $i\in [6, 23]$. Thus, the commutation
    relations among the $x_i$'s given in Lemma \ref{x's as commutators}
    translate into commutation relations among the $y_i$'s. In particular, we
    have $y_9 = \frac{1}{[2]_q} [y_8, y_{10}]$, $y_{13} = [y_{10}, y_{17}]$,
    $y_{14} = [y_{12}, y_{17}]$, $y_{15} = \frac{1}{[2]_q} [y_{14}, y_{17}]$,
    $y_{16} = \frac{1}{[2]_q} [y_{14}, y_{19}]$, $y_{18} = \frac{1}{[2]_q}
    [y_{17}, y_{19}]$, $y_{20} = [y_{19}, y_{21}]$, and $y_{22} =
    \frac{1}{[2]_q} [y_{21}, y_{23}]$.

    Next we apply Proposition \ref{proposition 1} to identify which of the
    $y_i$'s correspond to the standard Chevalley generators $E_i$.  We get $y_1
    = E_4$, $y_2 = E_1$, $y_{21} = E_3$, and $y_{24} = E_2$. We next identify
    the $y_i$'s that can be written as $q$-commutators of these $E_i$'s. For
    example, we have $y_3 = \mathbf{T}_{\underline{41}}(E_2) =
    \mathbf{T}_{1\cdot 4}(E_2) = \mathbf{T}_1(E_2) = E_{12} = [y_2, y_{24}]$.
    We note here we have adopted the same underlining notation, as in
    $\mathbf{T}_{\underline{41}}$ above, as well as the dot notation, as in
    $\mathbf{T}_{1 \cdot 4}$ above, used in the proof of Lemma \ref{x's as
    commutators}.  With this, we also have $y_{17} = \mathbf{T}_{4\cdot
    123421323124321}(E_3) = \mathbf{T}_{4}(E_3) = E_{43} = [y_1, y_{21}]$ and
    $y_{23} = \mathbf{T}_{\underline{4123421323124321323432}}(E_3) =
    \mathbf{T}_{32\cdot 12321432341232143234}(E_3) = \mathbf{T}_{32}(E_3) =
    E_{32} = [y_{21}, y_{24}]$.

    Now that we have established some of the identities of this lemma, we can
    continue with this same strategy to establish further identities. We have
    \[
        \begin{split}
            y_5 &= \mathbf{T}_{\underline{412}3}(E_4) = \mathbf{T}_{12\cdot
            43}(E_4) = \mathbf{T}_{12}(E_3) = E_{123} = [y_3, y_{21}],
            \\
            y_{19} &= \mathbf{T}_{4\underline{12342132312432132}}(E_3) =
            \mathbf{T}_{432\cdot 123214323412321}(E_3) = \mathbf{T}_{4\cdot
            32}(E_3) = \mathbf{T}_4(E_{32})
            \\
            &= E_{432} = [y_{17}, y_{24}],
            \\
            y_4 &= \mathbf{T}_{41\cdot 2}(E_3) = \mathbf{T}_{41}([E_2, E_3]) =
            [\mathbf{T}_{4\cdot 1}(E_2), \mathbf{T}_{4\cdot 1}(E_3)] =
            [\mathbf{T}_4(E_{12}), \mathbf{T}_4(E_3)]
            \\
            &= [E_{12}, E_{43}] = [y_3, y_{17}].
        \end{split}
    \]

    Next we compute a few more identities that build off of the identities
    already established. We have
    \begin{align*}
        [2]_q y_6 &= [2]_q \mathbf{T}_{4123\cdot 4}(E_2) = [2]_q
        \mathbf{T}_{41\cdot 23}(E_2) =  \mathbf{T}_{41}(E_{233})
        \\
        &= [\mathbf{T}_{41}(E_{23}), \mathbf{T}_{4\cdot 1}(E_3)] = [y_4,
        \mathbf{T}_4(E_3)] = [y_4, y_{17}],
        \\
        y_7 &= \mathbf{T}_{41234\cdot 2}(E_1) = \mathbf{T}_{41234}([E_2, E_1])
        = [\mathbf{T}_{41234}(E_2), \mathbf{T}_{41234}(E_1)] = [y_6, E_2]
        \\
        &= [y_6, y_{24}],
        \\
        y_8 &= \mathbf{T}_{412342\cdot 1}(E_3) =
        \mathbf{T}_{4123\underline{42}}(E_3) = \mathbf{T}_{41232\cdot 4}(E_3) =
        \mathbf{T}_{41232}([E_4, E_3]) \\
        &= [\mathbf{T}_{4123\cdot 2}(E_4), \mathbf{T}_{4\cdot 1232}(E_3)] =
        [\mathbf{T}_{4123\cdot 2}(E_4), \mathbf{T}_{4\cdot 1232}(E_3)]
        \\
        &= [\mathbf{T}_{4123}(E_4), \mathbf{T}_4(E_3)] = [y_5, y_{17}],
        \\
        y_{10} &= \mathbf{T}_{4123421\cdot 32}(E_3) =
        \mathbf{T}_{4123421}([E_3, E_2]) = [\mathbf{T}_{412342\cdot 1}(E_3),
        \mathbf{T}_{4123421}(E_2)]
        \\
        &= [\mathbf{T}_{412342}(E_3), E_2] = [y_8, y_{24}],
        \\
        [2]_q y_{11} &= [2]_q \mathbf{T}_{\underline{412342}1323}(E_1) = [2]_q
        \mathbf{T}_{123\cdot 4321323}(E_1) = [2]_q \mathbf{T}_{1\cdot 23}(E_2)
        \\
        &= \mathbf{T}_{1}([E_{23}, E_3]) = [\mathbf{T}_{1}(E_{23}),
        \mathbf{T}_1(E_3)] = [E_{123}, E_{3}] = [y_5, y_{21}],
        \\
        y_{12} &= \mathbf{T}_{\underline{41234213231}}(E_2) =
        \mathbf{T}_{1232\cdot 4321323}(E_2) = \mathbf{T}_{123\cdot 2}(E_1) =
        \mathbf{T}_{123}([E_2, E_1])
        \\
        &= [\mathbf{T}_{123}(E_2), \mathbf{T}_{123}(E_1)] =
        [\mathbf{T}_{123}(E_2), E_2] = [y_{11}, y_{24}].
    \end{align*}
    Finally, we can use $q$-associativity to establish the remaining
    identities,
    \begin{align*}
        &[2]_q y_7 = [2]_q [y_6, y_{24}] = [[y_4, y_{17}], y_{24}] = [y_4,
        [y_{17}, y_{24}]] = [y_4, y_{19}],
        \\
        &y_{19} = [[E_4, E_3], E_2] = [E_4, [E_3, E_2]] = [y_1, y_{23}],
        \\
        &y_{10} = [y_8, y_{24}] = [[y_5, y_{17}], y_{24}] = [y_5, [y_{17},
        y_{24}] = [y_5, y_{19}],
        \\
        &[2]_q y_{12} = [2]_q [y_{11}, y_{24}] = [[y_5, y_{21}], y_{24}] =
        [y_5, [y_{21}, y_{24}]] = [y_5, y_{23}],
    \end{align*}

\end{proof}

\section{Quantum Symmetric Matrices}

The algebra of $n \times n$ quantum symmetric matrices \cite{Kamita, Noumi} is
a quantized nilradical ${\mathcal U}_q(\mathfrak{n}_J)$ for the case when the
underlying Lie algebra $\mathfrak{g}$ is of type $C_n$ and $J =
\left\{\alpha_n\right\}$.  In this section, we prove that Conjecture \ref{conj}
holds in this case.

We let $x_{ij}$ ($1 \leq i \leq j \leq n$) denote the standard generators of
the algebra of quantum symmetric matrices. The defining relations are given in
\cite[Proposition 5.2]{Kamita}.  Let ${\mathcal N}_J$ be the normal subalgebra
of ${\mathcal U}_q(\mathfrak{n}_J)$. It is generated by the normal elements
$\Theta_1,\dots, \Theta_n$. Here, the quantized nilradical ${\mathcal
U}_q(\mathfrak{n}_J)$ is $\mathbb{N}$-graded with $\operatorname{deg}(x_{ij}) =
1$ for all $1 \leq i \leq j \leq n$. In view of this,
$\operatorname{deg}(\Theta_i) = i$ (for $1 \leq i \leq n$). The simple roots
are $\alpha_i = e_i - e_{i + 1}$ (for $1\leq i < n$) and $\alpha_n = 2e_n$. The
fundamental weights are $\varpi_i = e_1 + \cdots + e_i$ (for $1\leq i \leq n$).
Let $Q = \mathbb{Z}\alpha_1 + \cdots + \mathbb{Z}\alpha_n$ and $P =
\mathbb{Z}\varpi_1 + \cdots + \mathbb{Z}\varpi_n$ be the root lattice and
weight lattice respectively, and let $\langle -,-\rangle$ be the symmetric
bilinear form on $P$ defined by the rule $\langle e_i, e_j \rangle =
\delta_{ij}$. The algebra ${\mathcal U}_q(\mathfrak{n}_J)$ is $Q$-graded with
\[
    \operatorname{deg}_Q(x_{ij}) = e_i + e_j = (\alpha_i + \alpha_{i + 1} +
    \cdots \alpha_n) + (\alpha_j + \alpha_{j + 1} + \cdots + \alpha_{n-1}).
\]
Consider the parabolic element $w_J : = w_o^Jw_o \in W$. We have the reduced
expression
\begin{equation}
    \label{reduced expression, q-symmetric}
    w_J = (s_n s_{n-1} \cdots s_1)(s_n s_{n-1} \cdots s_2) \cdots (s_n s_{n-1})
    s_n.
\end{equation}
We recall (\ref{commutation with normal subalgebra}) the commutation relations,
\[
    x_{ij}\Theta_k = q^{-\langle \operatorname{deg}_Q(x_{ij}), (1 +
    w_J)\varpi_k \rangle} \Theta_k x_{ij}
\]
for $1\leq i \leq j \leq n$ and $1 \leq k \leq n$.

\begin{proposition}
    \label{Proposition, q-symmetric}

    Suppose $\mathfrak{g}$ is the Lie algebra of type $C_n$ with $n > 1$, and
    suppose $J = \left\{\alpha_n\right\}$.  If $\psi$ is an automorphism of the
    quantized nilradical ${\mathcal U}_q(\mathfrak{n}_J)$, then $\psi(\Theta_i)
    \in \mathbb{K}^\times \Theta_i$ for every $i \in \left\{1,\dots,n\right\}$.

\end{proposition}

\begin{proof}

    As in the proof of Theorem \ref{Theorem, B6}, the hypotheses in Theorem
    \ref{when Aut is Gr-Aut} involving the core and the existence of relations
    of the form $xy = \kappa yx$ can be seen to be satisfied by observing the
    relevant properties of the reduced expression (\ref{reduced expression,
    q-symmetric}).  Hence, every automorphism of ${\mathcal
    U}_q(\mathfrak{n}_J)$ preserves the $\mathbb{N}$-grading.

    The normal subalgebra ${\mathcal N}_J$ is invariant under any algebra
    automorphism of ${\mathcal U}_q(\mathfrak{n}_J)$. Since
    $\operatorname{dim}_\mathbb{K}({\mathcal N}_J)_1 = 1$ (in fact, $({\mathcal
    N}_J)_1 = \mathbb{K}\Theta_1$), then $\psi(\Theta_1) \in \mathbb{K}^\times
    \Theta_1$.  Furthermore, since $\mathbb{K}x_{nn} = \left\{ x\in ({\mathcal
    U}_q(\mathfrak{n}_J))_1 : x\Theta_1 = q^2\Theta_1x \right\}$, $\psi(x_{nn})
    = \mathbb{K}^\times x_{nn}$.

    For $k \in \left\{1,\dots, n\right\}$, $\operatorname{dim}_{\mathbb{K}}
    \left( \left({\mathcal N}_J \right)_k \right) = {\mathcal P}(k)$, where
    ${\mathcal P}$ is the partition function.  We also have $x_{nn}\Theta_k =
    q^2 \Theta_k x_{nn}$ for every $k \in \left\{1,\dots, n - 1\right\}$.

    For a fixed natural number $k\in \mathbb{N}$, we represent a partition
    $\nu$ of $k\in \mathbb{N}$ (and write $\nu \vdash k$) by a weakly
    increasing sequence of natural numbers $\nu_1 \leq \nu_2 \leq \cdots$ such
    that $\sum_i \nu_i = k$. For a partition $\nu \vdash k$, we let
    $\operatorname{parts}(\nu)$ be the number of parts of $\nu$ and write $\nu
    = (\nu_1,\nu_2,\dots, \nu_{\operatorname{parts}(\nu)})$, and we define the
    monomial
    \[
        \Theta^\nu := \Theta_{\nu_1}\Theta_{\nu_2} \cdots
        \Theta_{\nu_{\operatorname{parts}(\nu)}} \in \left({\mathcal
        N}_J\right)_k.
    \]
    Suppose
    \[
        \psi(\Theta_k) = \sum_{\nu\vdash k} c_{\nu,k} \Theta^{\nu},
        \hspace{10mm} (1\leq k \leq n).
    \]
    for scalars $c_{\nu, k} \in \mathbb{K}$. For $1\leq k < n$, we apply the
    automorphism $\psi$ to the relation $x_{nn}\Theta_k = q^2\Theta_k x_{nn}$
    to conclude that
    \[
        \sum_{\nu \vdash k}c_{\nu, k} x_{nn}\Theta^\nu =
        \sum_{\nu \vdash k}q^2 c_{\nu, k} \Theta^\nu x_{nn}.
    \]
    However since $x_{nn}\Theta_k = q^2\Theta_kx_{nn}$ for every $k \in
    \left\{1,\dots., n-1\right\}$, then in the above sum we can replace
    $x_{nn}\Theta^\nu$ with $q^{2 \cdot \operatorname{parts}(\nu)}\Theta^\nu
    x_{nn}$. Thus,
    \[
        \sum_{\nu \vdash k}\left(q^{2 \cdot \operatorname{parts}(\nu)} -
        q^2\right) c_{\nu, k} \Theta^\nu x_{nn} = 0.
    \]
    Since $q$ is not a root of unity, the only nonzero coefficients $c_{\nu,
    k}$ appearing in the above sum are those such that
    $\operatorname{parts}(\nu) = 1$. In other words, there is at most one
    monomial $\Theta^\nu$ in the sum $\sum_{\nu \vdash k} c_{\nu, k}\Theta^\nu$
    with a nonzero coefficient, namely $\Theta_k$. Hence $\psi(\Theta_k) \in
    \mathbb{K}^\times \Theta_k$ for $k < n$.

    Finally, consider $\Theta_n$. Since $\Theta_n$ generates the center of
    ${\mathcal U}_q(\mathfrak{n}_J)$ (see e.g. \cite{J}), then $\psi(\Theta_n)
    \in \mathbb{K}^\times \Theta_n$.

\end{proof}

\begin{theorem}
    \label{Theorem, q-symmetric}

    If $\mathfrak{g}$ is the Lie algebra of type $C_n$ with $n > 1$ and $J =
    \left\{\alpha_n\right\}$ (i.e. ${\mathcal U}_q(\mathfrak{n}_J)$ is the
    algebra of $n\times n$ quantum symmetric matrices), then
    $\operatorname{Aut}({\mathcal U}_q(\mathfrak{n}_J)) \cong
    \left(\mathbb{K}^\times\right)^n$.

\end{theorem}

\begin{proof}

    Suppose $\phi$ is an automorphism of ${\mathcal U}_q(\mathfrak{n}_J)$. As
    established in Proposition \ref{Proposition, q-symmetric}, all hypotheses
    of Theorem \ref{when Aut is Gr-Aut} are satisfied. Hence, $\phi$ is a
    graded automorphism.

    Next we will apply Theorem \ref{root vector, fixed, 2} to show that each
    $x_{ij}$ with $i + j \neq n + 1$ gets sent to a multiple of itself by
    $\phi$. First, suppose $i$ and $j$ are chosen such that $1 \leq i \leq j
    \leq n$ and $i + j < n + 1$, and consider the corresponding element
    $x_{ij}$. We show that for every other Lusztig root vector $x_{k\ell}$,
    there is a normal element $\Theta_p$ such that $x_{ij}$ and $x_{k\ell}$
    commute differently with $\Theta_p$. Equivalently, this means $\langle (e_i
    + e_j) - (e_k + e_\ell), (1 + w_J)\varpi_p \rangle \neq 0$. With this,
    Theorem \ref{root vector, fixed, 2} implies $x_{ij}$ gets sent to a
    multiple of itself by the automorphism $\phi$. There are two cases to
    consider. If $i \neq k$, let $p = \operatorname{min}(i, k)$. On the other
    hand, if $i = k$, let $p = \operatorname{min}(j, \ell)$. Now suppose $i$
    and $j$ are chosen so that $1 \leq i \leq j \leq n$ and $i + j > n + 1$,
    and consider the corresponding Lusztig root vector $x_{ij}$. As before, we
    will show that for every other Lusztig root vector $x_{k\ell}$, there is a
    normal element $\Theta_p$ such that $x_{ij}$ and $x_{k\ell}$ commute
    differently with $\Theta_p$. Here, if $j \neq \ell$, put $p =
    \operatorname{max}(j - 1, \ell - 1)$. However, if $j = \ell$, put $p =
    \operatorname{max}(i - 1, k - 1)$.

    Now we will show that every Lusztig root vector $x_{ij}$ with $i + j = n +
    1$ gets sent to a multiple of itself by $\phi$. Theorem \ref{root vector,
    fixed, 2} does not apply here because these $x_{ij}$ all commute the same
    way with each normal element $\Theta_p$. In fact, each of these $x_{ij}$
    commute with all of the elements in the normal subalgebra. None of the
    other $x_{ij}$'s behave this way. This means $S :=
    \operatorname{span}_{\mathbb{K}}\left\{x_{ij} \mid i + j = n + 1\right\}$
    is a $\phi$-invariant vector subspace of ${\mathcal U}_q(\mathfrak{n}_J)_1$
    For every $y \in S$, define
    \[
        C(y) := \left\{ x\in \left({\mathcal U}_q(\mathfrak{n}_J)\right)_1
        \mid yx = q xy \right\}.
    \]
    Each $C(y)$ is a vector space. Observe $\operatorname{dim}(C(y)) =
    \operatorname{dim}(C(\phi(y))$.  From the defining relations of ${\mathcal
    U}_q(\mathfrak{n}_J)$, we obtain that $C(x_{ij})$ is spanned by
    $\left\{x_{ik} : k < j\right\} \cup \left\{x_{kj} : k < i\right\}$. Hence
    $\operatorname{dim}(C(x_{ij})) = n - i$. The only elements $y\in S$ with
    $\operatorname{dim}(C(y)) = n - i$ are the nonzero multiples of $x_{ij}$.
    Hence $\phi(x_{ij}) \in \mathbb{K}^\times x_{ij}$.

    We have shown now that $\phi$ is a diagonal automorphism. Since ${\mathcal
    U}_q(\mathfrak{n}_J)$ has rank $n$ as a CGL extension, \cite[Theorems 5.3
    and 5.5]{GY2} imply that $\operatorname{Aut}\left( {\mathcal U}_q \left(
    \mathfrak{n}_J \right) \right) \cong \left( \mathbb{K}^\times \right)^n$.

\end{proof}


\begin{thebibliography}{20}\frenchspacing

    \bibitem{AC} J. Alev and M. Chamarie, \textit{D\'erivations et
        automorphismes de quelques alg\`ebres quantiques}, Comm. Algebra, {\bf
        20}, (1992), no. 6, 1787 - 1802.

    \bibitem{AD} N. Andruskiewitsch and F. Dumas, \textit{On the automorphisms
            of ${\mathcal U}_q^+\left( \mathfrak{g} \right)$}, In: Quantum
            groups, 107–133, IRMA Lect. Math. Theor. Phys. 12, Eur. Math. Soc.,
            Z\"urich, 2008.

    \bibitem{C} G. Cauchon, \textit{Effacement des d\'erivations et spectres
        premiers d'alg\`ebres quantiques}, J. Algebra, (2003), {\bf 260}, no.
        2, 476-518.

    \bibitem{CPWZ1} S. Ceken, J. H. Palmieri, Y.-H. Wang, and J.J. Zhang,
        \textit{The discriminant controls automorphism groups of noncommutative
        algebras}, Adv. Math., {\bf 269}, (2015), 551–584.

    \bibitem{CPWZ2} S. Ceken, J. H. Palmieri, Y.-H. Wang, and J.J. Zhang,
        \textit{The discriminant criterion and automorphism groups of quantized
        algebras}, Adv. Math., {\bf 286}, (2016), 754–801.

    \bibitem{CGWZ} K. Chan, J. Gaddis, R. Won, and J.J. Zhang,
        \textit{Reflexive hull discriminants and applications}, Selecta Math.,
        {\bf 28}, no. 40, (2022), https://doi.org/10.1007/s00029-021-00755-x

    \bibitem{DKP} C. De Concini, V. Kac, and C. Procesi, \textit{Some quantum
        analogues of solvable Lie groups}, In: Geometry and analysis (Bombay,
        1992), pp. 41-65, Tata Inst. Fund. Res., Bombay, 1995.

    \bibitem{FRT} L. D. Faddeev, N. Yu. Reshetikhin, and L. A. Takhtadzhyan,
        \textit{Quantization of Lie groups and Lie algebras}, Leningrad Math.
        J., {\bf 1}, (1990), 193-225.

    \bibitem{GKM} J. Gaddis, E. Kirkman, and W. F. Moore. \textit{On the
        discriminant of twisted tensor products}. J. Algebra, {\bf 477},
        (2017), 29–55.

    \bibitem{GL} J. Gaddis and T. Lamkin, \textit{Centers and automorphisms of
        PI quantum matrix algebras}, arxiv:2207.11956, to appear in
        \textit{Contemporary Mathematics}.

    \bibitem{GLS} C. Gei\ss, B. Leclerc, and J. Schr\"oer, \textit{Cluster
        structures on quantum coordinate rings}, Selecta Math. (N.S.), {\bf
        19}, (2013), 337 - 397.

    \bibitem{Goodearl Letzter} K. Goodearl and E. Letzter, \textit{The
        Dixmier–Moeglin equivalence in quantum coordinate rings and quantized
        Weyl algebras}, Trans. Amer. Math. Soc., (2000) {\bf 352}, no. 3,
        1381-1403.

    \bibitem{Goodearl Yakimov} K. Goodearl and M. Yakimov, \textit{Quantum
        cluster algebra structures on quantum nilpotent algebras}, Memoirs
        Amer. Math. Soc., {\bf 247} (2017), no. 1169, vii + 119pp.

    \bibitem{GY1} K. Goodearl and M. Yakimov, \textit{Unipotent and Nakayama
        automorphisms of quantum nilpotent algebras},in: Commutative Algebra
        and Noncommutative Algebraic Geometry II, eds: D. Eisenbud et al, pp
        181-212, MSRI Publ. Vol 68, Cambridge Univ. Press, 2015.

    \bibitem{GY2} K. Goodearl and M. Yakimov, \textit{From quantum Ore
        extensions to quantum tori via noncommutative UFDs}, Adv. Math., {\bf
        300}, (2016), 672 - 716.

    \bibitem{HW} N. Hu, and X. Wang, \textit{Convex PBW-type Lyndon bases and
        restricted two-parameter quantum groups of type $G_2$}, Pacific J.
        Math., {\bf 241}, (2009), no. 2, 243 - 273.

    \bibitem{Humphreys} J.E. Humphreys, \textit{Introduction to Lie algebras
        and representation theory}, Spring-Verlag, New York, 1972.

    \bibitem{J} H. P. Jakobsen, \textit{The center of ${\mathcal
        U}_q\left(\mathfrak{n}_{\omega}\right)$}, Comm. Algebra, {\bf 46},
        (2018), no. 1, 262 - 282.

    \bibitem{Jantzen} J.C. Jantzen, \textit{Lectures on Quantum groups},
        Graduate Texts in Mathematics 6, Amer. Math. Soc., 1996.

    \bibitem{JJ} A. Jaramillo and G. Johnson, \textit{Quantized nilradicals of
            parabolic subalgebras of $\mathfrak{sl}(n)$ and algebras of
            coinvariants}, Comm. Alg., (2022), {\bf 50}, no. 11, 4997-5015.

    \bibitem{Kamita} A. Kamita, \textit{Quantum deformations of prehomogeneous
        vector spaces III}, Hiroshima Math. J., {\bf 30}, (2000), 79 - 115.

    \bibitem{Kashiwara} M. Kashiwara, \textit{Crystalizing the $q$-analogue of
        universal enveloping algebras}, Commun. Math. Phys., {\bf 133}, (1990),
        249 - 260.

    \bibitem{LL} S. Launois and T. H. Lenagan, \textit{Primitive ideals and
        automorphisms of quantum matrices}, Algebras and Rep. Theory 10 (2007),
        339–365.

    \bibitem{LL2} S. Launois and T. H. Lenagan, \textit{Automorphisms of
        quantum matrices}, Glasgow Math. J., {\bf 55(A)}, (2013), 89-100.

    \bibitem{LS} S.Z. Levendorskii and Yan Soibelmann, \textit{Algebras of
        functions on compact quantum groups, Schubert cells and quantum tori},
        Comm. Math. Phys., {\bf 139}, (1991), 141-170.

    \bibitem{L} G. Lusztig, \textit{Introduction to quantum groups}, Progr.
        Math. 110, Birkh\"auser, 1993.

    \bibitem{Lusztig} G. Lusztig, \textit{Canonical bases arising from
        quantized enveloping algebras}, J. Amer. Math. Soc., {\bf 3}, (1990),
        447 - 498.

    \bibitem{MC} A. M\'eriaux and G. Cauchon, \textit{Admissible diagrams in
        $U_q^w(\mathfrak{g})$ and combinatoric properties of Weyl groups},
        Represent. Theory, {\bf 14}, (2010), 645 - 687.

    \bibitem{Musson} I. M. Musson, \textit{Ring theoretic properties of the
        coordinate rings of quantum symplectic and Euclidean space}, in Ring
        Theory, Proc. Biennial Ohio State-Denison Conf. 1992, S. K. Jain and S.
        T. Rizvi, eds., World Scientific, Singapore, (1993), 248-258.

    \bibitem{Noumi} M. Noumi, \textit{Macdonald's symmetric polynomials as
        zonal spherical functions on some quantum homogeneous spaces}, Advances
        in Math., {\bf 123}, (1996), 16-77.

    \bibitem{Strickland} E. Strickland, \textit{Classical invariant theory for
        the quantum symplectic group}, Advances in Math. {\bf 123}, (1996),
        78-90.

    \bibitem{Y} M. Yakimov, \textit{Invariant prime ideals in quantizations of
        nilpotent Lie algebras}, Proc. London Math. Soc. (3) {\bf 101} (2010),
        no. 2, 454-476.

    \bibitem{Y2} M. Yakimov, \textit{On the spectra of quantum groups}, Mem.
        Amer. Math. Soc., {\bf 229}, (2014), no. 1078, vi+91.

    \bibitem{Y3} M. Yakimov, \textit{Rigidity of quantum tori and the
        Andruskiewitsch-Dumas conjecture}, Selecta Math. {\bf 20}, (2014), no.
        2, 421–464.

    \bibitem{Y4} M. Yakimov, \textit{The Launois-Lenagan conjecture}, J.
        Algebra, {\bf 392}, (2013), 1-9.



\end{thebibliography}
\end{document}